\documentclass[11pt]{amsart}
\usepackage{amsfonts,latexsym,amsthm,amssymb,amsmath,amscd,euscript,tikz, tikz-cd}
\usepackage[alphabetic, msc-links, bibtex-style, nobysame]{amsrefs}
\usepackage{stackengine}
\usepackage{framed}
\usepackage{xfrac}
\usepackage[makeroom]{cancel}
\usepackage{faktor}
\usepackage{braket}
\usepackage{pgf,tikz,pgfplots}
\pgfplotsset{compat=1.17}
\usepgfplotslibrary{fillbetween} 
\usepackage{mathrsfs}
\usetikzlibrary{arrows}
\definecolor{wrwrwr}{rgb}{0.3803921568627451,0.3803921568627451,0.3803921568627451}
\definecolor{rvwvcq}{rgb}{0.08235294117647059,0.396078431372549,0.7529411764705882}
\definecolor{mblue}{rgb}{0.2, 0.3, 0.8}
\definecolor{morange}{rgb}{1, 0.5, 0}
\definecolor{mgreen}{rgb}{0.1, 0.4, 0.2}
\definecolor{mred}{rgb}{0.5, 0, 0}
\definecolor{ForestGreen}{RGB}{34,139,34}
\usepackage{hyperref}
\hypersetup{colorlinks=true,citecolor=ForestGreen,linkcolor = blue,urlcolor =black,linkbordercolor={1 0 0}}
\usepackage{float}

\usepackage{lmodern}
\usepackage{supertabular}
\usepackage{verbatim}
\usepackage{amssymb}
\usepackage{enumerate}
\usepackage{stmaryrd}
\usepackage{bbm}
\usepackage{mathtools}
\usepackage[nopatch=footnote]{microtype}
\usepackage{setspace}
\usepackage{tikz,tikz-cd}
\usetikzlibrary{matrix,calc,positioning,arrows,decorations.pathreplacing,patterns,knots}

\newcommand{\la}{\langle}
\newcommand{\rg}{\rangle}

\newcommand{\smallmres}{\mathbin{\vrule height 0.8ex depth 0pt width
0.065ex\vrule height 0.065ex depth 0pt width 0.65ex}}
\newcommand{\mres}{\mathbin{\vrule height 1.6ex depth 0pt width
0.13ex\vrule height 0.13ex depth 0pt width 1.3ex}}

\newtheorem{theorem}{{Theorem}}[section]
\newtheorem*{theorem*}{Theorem}
\newtheorem{lemma}[theorem]{Lemma}
\newtheorem{proposition}[theorem]{Proposition}

\newtheorem*{corollary*}{Corollary}

\theoremstyle{definition}
\newtheorem{definition}{Definition}

\newtheorem{thm}{Theorem}
\newtheorem{cor}{Corollary}

\usepackage{lmodern,url,enumerate,mathtools,microtype}
\usepackage[hmargin = 1in,vmargin=1in]{geometry}
\usepackage{graphicx}
\usepackage{subcaption}

\newcommand{\ve}{\varepsilon}

\newcommand{\mr}[1]{{\rm #1}}
\newcommand{\cB}{\mathcal{B}}

\newcommand{\cF}{\mathcal{F}}
\newcommand{\cG}{\mathcal{G}}\newcommand{\cH}{\mathcal{H}}
\newcommand{\cJ}{\mathcal{J}}
\newcommand{\cL}{\mathcal{L}}

\newcommand{\cP}{\mathcal{P}}
\newcommand{\cR}{\mathcal{R}}

\newcommand{\bG}{\mathbb{G}}

\newcommand{\bN}{\mathbb{N}}

\newcommand{\bR}{\mathbb{R}}
\newcommand{\bS}{\mathbb{S}}

\newcommand{\nc}{\newcommand}

\nc{\on}{\operatorname}
\nc{\p}{\partial}
\nc{\ol}{\overline}
\nc{\ul}{\underline}
\nc{\pa}{\partial}

\nc{\pb}{\partial_b}
\nc{\pc}{\partial_c}
\nc{\pd}{\partial_d}
\nc{\pe}{\partial_e}
\nc{\pf}{\partial_f}
\nc{\pg}{\partial_g}
\nc{\ph}{\partial_h}
\nc{\pari}{\partial_i}
\nc{\pj}{\partial_j}
\nc{\pk}{\partial_k}
\nc{\pl}{\partial_l}
\nc{\pell}{\partial_\ell}
\nc{\parm}{\partial_m}
\nc{\pn}{\partial_n}
\nc{\po}{\partial_o}
\nc{\pp}{\partial_p}
\nc{\pq}{\partial_q}
\nc{\pr}{\partial_r}
\nc{\ps}{\partial_s}
\nc{\pt}{\partial_t}
\nc{\pu}{\partial_u}
\nc{\pv}{\partial_v}
\nc{\pw}{\partial_w}
\nc{\px}{\partial_x}
\nc{\py}{\partial_x}
\nc{\pz}{\partial_z}

\nc{\Spec}{\on{Spec}}

\nc{\sn}{\mr{sn}}
\nc{\cn}{\mr{cn}}
\nc{\dn}{\mr{dn}}

\allowdisplaybreaks

\newcommand{\@dotsep}{4.5}

\def\@tocline#1#2#3#4#5#6#7{%
  \relax
  \ifnum #1>\c@tocdepth
  \else
    \par \addpenalty\@secpenalty \addvspace{#2}%
    \begingroup
      \hyphenpenalty\@M
      \@ifempty{#4}{%
        \@tempdima\csname r@tocindent\number#1\endcsname\relax
      }{%
        \@tempdima#4\relax
      }%
      \parindent\z@
      \leftskip#3\relax
      \advance\leftskip\@tempdima\relax
      \rightskip\@pnumwidth plus4em
      \parfillskip-\@pnumwidth
      #5\leavevmode\hskip-\@tempdima #6\nobreak
      \leaders\hbox{$\m@th
        \mkern \@dotsep mu.\mkern \@dotsep mu$}\hfill
      \nobreak\hbox to\@pnumwidth{\@tocpagenum{#7}}\par
    \endgroup
  \fi
}

\renewcommand{\l@section}{\@tocline{1}{2pt}{0pt}{}{\bfseries}}
\renewcommand{\l@subsection}{\@tocline{2}{0pt}{1.5em}{3.2em}{}}
\renewcommand{\l@subsubsection}{\@tocline{3}{0pt}{4.7em}{3.8em}{}}

\makeatother

\title[The sharp isoperimetric inequality for varifolds in codimension two]{The sharp isoperimetric inequality for varifolds in codimension two}
\date{\today}

\author[Raphael Tsiamis]{Raphael Tsiamis}
\address{Department of Mathematics, Columbia University
\newline {\href{mailto:r.tsiamis@columbia.edu}{r.tsiamis@columbia.edu}}}

\begin{document}

\begin{abstract}
We prove a Michael-Simon inequality for a class of varifolds in Euclidean space that includes all integral varifolds with locally bounded first variation of arbitrary dimension and codimension.
In codimension $2$, this inequality is sharp and implies the sharp isoperimetric inequality for varifolds, which attains equality precisely for the round disk.
Our proof uses the optimal transportation approach of Brendle-Eichmair.
\end{abstract}

\vspace*{0.2in}
\maketitle

%\vspace{-0.2in}

\section{Introduction}

The isoperimetric inequality is one of the most fundamental geometric inequalities and has played an important role in the development of geometric measure theory.
In his groundbreaking work, Brendle~\cite{brendle-sharp-isoperimetric} proved a sharp Michael-Simon inequality for submanifolds in Euclidean space in codimension at most two, and in particular obtained the sharp isoperimetric inequality for minimal submanifolds.
Brendle's approach utilized ideas from optimal transportation, developed further in important work with Eichmair~\cite{brendle-eichmair}.
Optimal transportation has also played an important role in the modern theory of isoperimetric inequalities more broadly, including sharp quantitative forms of the anisotropic isoperimetric inequality due to Figalli, Maggi and Pratelli~\cite{figalli-maggi-pratelli}.

A central feature of geometric measure theory is that meaningful forms of the isoperimetric inequality continue to hold for weak notions of surfaces.
In their foundational work on normal and integral currents, Federer and Fleming~\cite{federer-fleming} established an isoperimetric inequality for integral currents, which became one of the basic tools in the compactness theory and the solution of Plateau-type problems. Almgren subsequently proved the optimal Euclidean isoperimetric inequality for integral currents~\cite{almgren-isoperimetric}, with the sharp Euclidean constant and the corresponding rigidity statement: equality is attained by the round sphere and its flat disk filling. In this setting, the isoperimetric inequality is naturally formulated as a filling inequality, relating the mass of a cycle to the least mass of a current which it bounds.

For varifolds, the natural replacement of the boundary is the distributional first variation.
Allard's first variation theory~\cite{allard-first-variation} includes a Sobolev and isoperimetric inequality for rectifiable varifolds with density bounded from below.
The closely related Michael-Simon inequality~\cite{michael-simon} gives a dimensionally natural Sobolev inequality in terms of the mean curvature and, in the smooth setting, the boundary of the submanifold. 
Our first main result is the following.
\begin{thm}\label{thm:varifold-isoperimetric}
    Let $n \geq 2$ and $m \geq 2$ be integers.
    Let $V$ be a compactly supported integral $n$-varifold with locally bounded first variation. 
    Then, it holds that
    \[
    n \Bigl( \frac{(n+m) \, |B^{n+m}|}{m |B^m|} \Bigr)^{\frac{1}{n}} \| V\| (\bR^{n+m})^{\frac{n-1}{n}} \leq \| \delta V\| (\bR^{n+m}).
    \]
\end{thm}
Our argument employs an optimal transportation approach inspired by the technique of Brendle and Eichmair~\cite{brendle-eichmair}.
For $m=2$, the recursion formula for the volume of Euclidean balls gives $|B^{n+2}| = \frac{2 \pi}{n+2} |B^n|$, so $\frac{2 |B^2|}{(n+2) |B^{n+2}|} = \frac{1}{|B^n|}$, so the inequality of Theorem~\ref{thm:varifold-isoperimetric} is sharp for varifolds of codimension $2$ in Euclidean space.
In particular, we obtain the following rigidity result.
\begin{cor}\label{cor:sharp-isoperimetric}
Let $V$ be a compactly supported integral $n$-varifold in $\bR^{n+2}$ satisfying
\[
n |B^n|^{\frac{1}{n}} \|V\| ( \bR^{n+2})^{\frac{n-1}{n}} = \| \delta V\| ( \bR^{n+2})
\]
as well one of the following properties:
\begin{enumerate}[(i)]
    \item the tangent plane map $x \mapsto T_x V$ is weakly differentiable on $V$; or
    \item the Radon measure $\cH^{n-1} \mres \textup{sing} \, V$ is absolutely continuous with respect to $\| \delta V\|$.
\end{enumerate}
Then, $V$ is an $n$-dimensional flat round disk.
\end{cor}

Theorem~\ref{thm:varifold-isoperimetric} is obtained by establishing a Michael-Simon inequality for integral varifolds, which is presented in Theorem~\ref{thm:optimal-MSS} and is sharp in codimension up to $2$.
In particular, these results imply the sharp isoperimetric inequality obtained by Brendle~\cites{brendle-sharp-isoperimetric , brendle-eichmair } when $\Sigma^n \subset \bR^{n+m}$ is a smooth compact submanifold with boundary and $V = |\Sigma|$.
In this case, $|V| = \cH^n \mres \Sigma$ has singular first variation $|\delta_s V| = \cH^{n-1} \mres \partial \Sigma$, so $\| \delta V\| = |H| \, \cH^n \mres \Sigma + \cH^{n-1} \mres \partial \Sigma$ and Theorem~\ref{thm:varifold-isoperimetric} becomes
\[
n \Bigl( \frac{(n+m) |B^{n+m}|}{m |B^m|} \Bigr)^{\frac{1}{n}} |\Sigma|^{\frac{n-1}{n}} \leq |\partial \Sigma| + \int_{\Sigma} |H|.
\]
Corollary~\ref{cor:sharp-isoperimetric} establishes the rigidity of the sharp isoperimetric inequality under natural measure-theoretic conditions.
The property $(i)$ means that $V$ is a curvature varifold with boundary, a class introduced by Hutchinson and Mantegazza~\cites{hutchinson , mantegazza} to study curvature-dependent variational problems, with compactness properties analogous to Allard's theorem.
The condition $(ii)$ allows arbitrary singular behavior detected by the first variation, and is also preserved along convergent sequences of varifolds with uniform $L^p$ bounds on the distributional mean curvature and a uniform positive lower bound on the $(n-1)$-dimensional density of $\delta V$ at non-smooth points.

The Michael-Simon inequality for varifolds is a fundamental tool throughout geometric analysis, providing the natural replacement for obtaining lower density bounds and compactness, regularity, and rectifiability results when a monotonicity formula is not available~\cites{dephilippis-pigati , anisotropic-michael-simon , anisotropic-surfaces , anisotropic-allard }.
The techniques presented here, with small modifications, can be used to obtain the sharp Wulff inequality for varifolds and characterize its equality case~\cite{du-yi-zhao}.
Previously, De Rosa-Kolasi\'nski-Santilli~\cite{derosa-kolasinski-santilli} characterized the equality case for the codimension-one inequality among finite perimeter sets, proving a conjecture of Maggi.
After this manuscript had been completed, the interesting work of Han-Xu~\cite{optimal-transport-codim-4} presented an alternative approach to obtain a sharp isoperimetric inequality and a sharp Wulff inequality in codimension up to four.
Independently, Gaudet~\cite{gaudet} recently obtained related inequalities, using alternative methods, under the additional assumption of weak twice differentiability and $L^2_{\textup{loc}}$ mean curvature.
We also highlight earlier work of Menne and Scharrer~\cites{menne-isoperimetric , menne-scharrer}, who established isoperimetric and Sobolev inequalities in a general varifold framework, including weak surfaces exhibiting diffuse behavior.

\subsection{Strategy of the proof}
We obtain Theorem~\ref{thm:varifold-isoperimetric} through a more general Michael-Simon inequality, established as Theorem~\ref{prop:fancy-integration-by-parts}.
To prove this result, we will formulate an optimal transportation scheme for Lipschitz functions on varifolds, following the approach of Brendle and Eichmair~\cites{brendle-eichmair , notices }.
Many of the techniques in their work transfer to our setting, with the important exception of the integration-by-parts step used to transfer an a.e.~ inequality for the Laplacian, obtained in Proposition~\ref{prop:laplacian-inequality}, to an integral inequality.
Significant additional care is required for this step, and we establish the necessary integration-by-parts inequality for semiconvex functions on $C^2$ rectifiable varifolds in Section~\ref{section:semiconvex}.
The resulting Proposition~\ref{prop:fancy-integration-by-parts} has independent interest and relies on a blowup technique for sub-level sets of convex functions inspired by the beautiful works of Collins-Tong~\cite{collins-tong} and Colding-Minicozzi~\cite{confined-in-space}.

Finally, in Section~\ref{section:rigidity} we analyze the equality case of the codimension-$2$ isoperimetric inequality for varifolds.
An important novelty from other related results, including the original argument of Brendle for submanifolds~\cite{brendle-sharp-isoperimetric}, is that we establish properties for the full subdifferential of the optimal transport potential, which we use to identify the varifold with a round $n$-disk.

\smallskip \noindent \textbf{Acknowledgments.}
I am very grateful to my advisor, Simon Brendle, for suggesting this problem, for many stimulating discussions, and for his invaluable guidance and continued support. 
I am also thankful to Antonio De Rosa, Benjy Firester, and Jonas Hirsch for helpful conversations.
This work was supported in part by the A.G.~Leventis Foundation Scholarship and the Onassis Foundation Scholarship.

\section{Semiconvex functions on \texorpdfstring{$C^2$}{ C\^2} rectifiable varifolds}\label{section:semiconvex}

Let $V$ be an integral $n$-varifold in $\bR^{n+m}$ with locally bounded first variation, so $\|V\|$-a.e.~ point $x$ admits an approximate tangent plane $T_x V$.
For a function $U(x)$ that is differentiable at $x$, we write $\nabla^V U(x) := \pi_{T_x V} DU(x)$ for the tangential gradient along $T_x V$.
Equivalently, if $x$ lies on a $C^1$ sheet $M$ representing $V$ with $T_x M = T_x V$, then $\nabla^V U(x)$ agrees with the tangential derivative of $U|_M$ at $x$ whenever the latter exists.
By Rademacher's theorem and rectifiability, every Lipschitz function $U$ has an approximate tangential derivative at $\|V\|$-a.e.~ point.
Because the first variation $\delta V$ is locally bounded, it is a vector-valued Radon measure with Lebesgue decomposition
\begin{equation}\label{eqn:HV-decomposition}
\delta V = - H_V \| V\| + \delta_s V, \qquad H_V \in L^1_{\textup{loc}}( \|V \| ; \bR^{n+m}), \qquad \delta_s V \perp \|V \|.
\end{equation}
The vector $H_V$ is the Radon-Nikodym derivative of the absolutely continuous part of $- \delta V$ with respect to $\| V \|$.
We call $H_V$ the generalized mean curvature of $V$, and $\delta_s V$ the singular first variation, so smooth submanifolds without boundary satisfy $\delta V(X) = - \int \la H,X \rg \, d \cH^n$ for $X$ any vector field.

We will formulate an optimal transportation scheme on the varifold $V$, in the style of Brendle-Eichmair~\cite{brendle-eichmair}.
An essential ingredient of our arguments will be a structural result of Menne~\cite{menne-second-order}*{Theorem 1}, combined with Brakke's orthogonality theorem~\cite{brakke}*{Chapter 5}, which shows that integral varifolds can be covered by countably many $C^2$ sheets whose mean curvature represents the distributional mean curvature $\|V\|$-a.e.
The subsequent discussion applies more generally to all varifolds satisfying these perpendicularity and locality properties for the mean curvature, including codimension-one varifolds with locally bounded anisotropic mean curvature and finite support measure, by Kolasi\'nski-Santilli~\cites{kolasinski-santilli , perpendicularity-and-locality }.

\begin{proposition}\label{prop:menne}
Let $V$ be an integral $n$-varifold with locally bounded first variation in an open set $\Omega \subset \bR^{n+m}$ and let $U : \bR^{n+m} \to \bR$ be a locally semiconvex, locally Lipschitz map.
Then, there exist countably many $n$-dimensional submanifolds $M_j \subset \Omega$ of class $C^2$, pairwise disjoint Borel sets $\cB_j \subset M_j \cap \textup{spt} \, \|V\|$, and a Borel function $\theta: \bigcup_j \cB_j \to \bN^*$, which satisfy 
\[
{ \textstyle \| V\| \mres \bigcup_j \cB_j = \|V\|, \qquad \|V\| \bigl( \Omega \setminus \bigcup_j \cB_j \bigr) = 0,} \qquad \|V\| \mres B_j = \theta \, \cH^n \mres \cB_j.
\]
Moreover, there exists a Borel set $\cG \subset \bigcup_j \cB_j$ with $\|V\| (\Omega \setminus \cG) = 0$ such that the following properties hold at every point $x \in \cG \cap \cB_j$ and every $C^2$ sheet $M_j$ through $x$.
\begin{enumerate}[(i)]
    \item The approximate tangent space $T_x V$ at $x$ agrees with $T_x M_j$, the classical mean curvature vector of $M_j$ at $x$ agrees with $H_V(x)$, and $x$ is a $\|V\|$-Lebesgue point of $H_V$.
    \item For any other sheet $M_i$ through $x$, the classical second fundamental forms $\mathrm{I\!I}_{M_i}(x)$ and $\mathrm{I\!I}_{M_j}(x)$ agree at $x$ and their common trace satisfies $\textup{tr} \, \mathrm{I\!I}_{M_i}(x) = \textup{tr} \, \mathrm{I\!I}_{M_j}(x) = H_V(x)$.
    Moreover, $U|_{M_i}$ and $U|_{M_j}$ both have an Alexandrov Hessian at $x$, and $D^2_{M_i} U(x) = D^2_{M_j} U(x)$.
    \item $x$ is a density point of $\|V\|$, namely $\|V\| \mres M_j$ has $\cH^n$-density equal to $\Theta^n( \|V\|,x)$ at $x$.
\end{enumerate}
\end{proposition}
\begin{proof}
By Menne's second order rectifiability theorem~\cite{menne-second-order}*{Theorem 1}, there is a countable family of $n$-dimensional $C^2$ submanifolds $M_j\subset\Omega$ which cover $V$ up to a $\|V\|$-null set and satisfy $H_{M_j}=H_V$ for $\|V\|$-a.e. point on each sheet.
Since $V$ is integral, we may write $\|V\|=\theta\,\cH^n\mres\cB$, where $\cB$ is countably $n$-rectifiable and $\theta=\Theta^n(\|V\|,-)\in\bN^*$ a.e., and then disjointify the cover by setting $\cB_j:=(\cB\cap M_j)\setminus\bigcup_{i<j}M_i$.
The Besicovitch differentiation theorem~\cite{maggi}*{Theorem 5.8} gives properties $(i)$ and $(iii)$ after discarding a $\|V\|$-null set.

It remains only to check the compatibility in $(ii)$ on overlaps.
For every pair $i,j$ with $\cH^n(M_i\cap M_j)>0$, almost every point of $M_i\cap M_j$ is a density point of the intersection relative to both sheets.
At such a point $x$, the tangent spaces agree.
Writing both sheets as $C^2$ graphs over this common tangent space, their difference vanishes on a density-one set; its second-order Taylor expansion therefore has vanishing linear and quadratic parts.
Hence $\mathrm{I\!I}_{M_i}(x)=\mathrm{I\!I}_{M_j}(x)$.
Similarly, by Alexandrov's theorem~\cite{villani}*{Theorems 14.1 and 14.25}, the restrictions of $U$ to both sheets admit second-order expansions at almost every such point, and the two expansions agree because the functions coincide on the same density-one set.
Thus $D^2_{M_i}U(x)=D^2_{M_j}U(x)$.
Discarding the resulting countable union of null sets proves the proposition.
\end{proof}

Let $x_0 \in \cB_j$ be a point of $\textup{spt} \, \|V\|$ satisfying the properties of Proposition~\ref{prop:menne}, with $\cB_j \subset M_j$ on a $C^2$ sheet representing $V$ near $x_0$.
We introduce a local coordinate chart near $x_0$.
We can find a positive real number $\rho_0$ and a $C^2$-diffeomorphism $\Xi^{M_j}_{x_0}$ from $\{ v \in T_{x_0} M_j : |v| \leq 2n \rho_0 \}$ onto an open subset of $M_j$, such that $\Xi^{M_j}_{x_0}(0) = x_0$ and $D \Xi^{M_j}_{x_0}(0) = \textup{id}_{T_{x_0} M_j}$.
After shrinking the coordinate chart if necessary, we may assume that $|\Xi^{M_j}_{x_0}(v) - x_0| \leq 2 \, |v|$ for all $v \in T_{x_0} M_j$ with $|v| < \rho_0$.
For $r \in (0, \rho_0)$, we take $\{ e_1, \dots, e_n \}$ to be an orthonormal basis of $T_{x_0} V = T_{x_0} M_j$, and define the cube
\begin{equation}\label{eqn:tangent-cube}
W_r := \{ v \in T_{x_0} M_j : |\la v, e_i \rg| \leq \tfrac{1}{2} r \; \text{ for all } \; i=1,\dots, n \} \subset T_{x_0} M_j
\end{equation}
with image $Q_r := \Xi^{M_j}_{x_0} ( W_r) \subset \{ x \in M_j : d_j(x, x_0) \leq \sqrt{n} \, r \}$ under the chosen $C^2$ diffeomorphism. 
\begin{lemma}\label{lemma:v-ae-q-gj}
    For $\|V\|$-a.e. point $x_0 \in \textup{spt} \, \|V\|$ satisfying properties $(i)$-$(iii)$ of Proposition~\ref{prop:menne}, there is a $C^2$ sheet $M_j$ through $x_0$, a Borel set $\cB_j \subset M_j$, and a $\rho_0>0$ such that for $r \in (0,\rho_0)$,
    \begin{equation}\label{eqn:full-density-point}
    \| V\| ( Q_r \setminus \cB_j) = o(r^n), \qquad \|V\|(Q_r \cap \cB_j) = \theta(x_0 )r^n + o(r^n)
\end{equation}
where $\theta(x_0) := \Theta^n( \|V\|, x_0)$.
In particular, $x_0$ is a density point of $\cB_j$ inside $M_j$.
Moreover,
\begin{equation}\label{eqn:density-approximation}
\lim_{r \downarrow 0} \frac{1}{r^n} \int_{Q_r \cap \cB_j} | \Theta^n( \|V\|, x) - \Theta^n (\|V\|, x_0) | \, d \cH^n (x) = 0.
\end{equation}
In particular, $x_0$ is an $\cH^n$-density point of $\cB_j$ relative to $M_j$ and an $L^1$-Lebesgue point on $M_j$ of the Radon-Nikodym derivative $\frac{d ( \|V\| \, \smallmres \, M_j)}{d ( \cH^n \, \smallmres \, M_j)}$, which equals $\Theta^n ( \|V\|,x_0)$.
Finally, the same properties~\eqref{eqn:full-density-point} and~\eqref{eqn:density-approximation} hold with the set $Q_r$ replaced by an ambient Euclidean ball $B_r(x_0)$.
\end{lemma}
\begin{proof}
Take $x_0$ to be simultaneously a density point of $\cB_j$ relative to $M_j$, a $\|V\|$-density point of $\cB_j$, and an $L^1$-Lebesgue point on $M_j$ of the Radon-Nikodym derivative $d(\|V\|\mres M_j)/d(\cH^n\mres M_j)$.
These properties hold for $\|V\|$-a.e. point by the Besicovitch differentiation theorem~\cite{maggi}*{Theorem 5.8}.
Since $Q_r$ is the image of a Euclidean cube of side length $r$ centered at the origin in $T_{x_0}M_j$ under the $C^2$ diffeomorphism $\Xi^{M_j}_{x_0}$, we have $\cH^n(Q_r) = r^n + o(r^n)$ and $B_{cr} \cap M_j \subset Q_r \subset B_{Cr} (x_0) \cap M_j$ for fixed $0 < c < C< \infty$ and all sufficiently small $r$.
The differentiation properties at $x_0$ therefore give~\eqref{eqn:full-density-point} and~\eqref{eqn:density-approximation} directly.
The same differentiation theorem gives the corresponding statements for ambient Euclidean balls.
This proves our assertion.
\end{proof}

Combining the properties of Proposition~\ref{prop:menne} and Lemma~\ref{lemma:v-ae-q-gj}, we introduce the following notion.

\begin{definition}\label{def:alexandrov-point}
Let the varifold $V$ and the locally semiconvex function $U$ be as in Proposition~\ref{prop:menne}.
    We call a point $x \in \on{spt} \|V\|$ a \textbf{good point} for $U$ if the following properties hold at $x$:
    \begin{enumerate}[(i)]
        \item $x$ lies on a $C^2$ sheet $M_j$ such that $T_x M_j = T_x V$ and the classical mean curvature of $M_j$ at $x$ equals $H_V(x)$.
        Also, $x$ is contained in the Borel set $\cB_j$ and satisfies the properties~\eqref{eqn:full-density-point} and~\eqref{eqn:density-approximation}. 
        \item The restriction of $\|V\|$ to $M_j$ has $\cH^n$-density equal to $\Theta^n( \|V\|,x)$ at $x$ and $x$ is a $\|V\|$-Lebesgue point of $H_V$ and $\frac{d ( \|V\| \, \smallmres \, M_j)}{d ( \cH^n \, \smallmres \, M_j)}$ relative to $\cB_j$, which satisfies Lemma~\ref{lemma:v-ae-q-gj}.
%        Moreover, the singular part of the distributional Laplacian $\Delta^{\textup{dist}}_{M_j} ( U|_{M_j})$ has zero $n$-density at $x$.
        \item $\nabla^V U(x)$ exists and is well-defined.
        \item For any other sheet $M_i$ through $x$, the second fundamental forms $\mathrm{I\!I}_{M_i}(x)$ and $\mathrm{I\!I}_{M_j}(x)$ agree.
        Moreover, the Alexandrov Hessians $D^2_{M_i} U(x)$ and $D^2_{M_j} U(x)$ are defined and agree.
    \end{enumerate}
    At a good point $ x \in \on{spt} \|V\|$, we choose any $C^2$ sheet $M_j$ representing $V$ at $x$.
    We define the varifold Hessian $D^2_V U(x)$ as the Alexandrov Hessian $D^2_{M_j} U(x)$.
    We define the second fundamental form $\mathrm{I\!I}_V(x)$ as the classical second fundamental form $\mathrm{I\!I}_{M_j}(x)$.
\end{definition}
By Proposition~\ref{prop:menne} and the Rademacher and Alexandrov differentiability theorems invoked above, the set of good points has full $\|V\|$-measure.
We now establish an important integration by parts inequality for semiconvex functions on mean curvature-rectifiable varifolds.

\begin{proposition}\label{prop:fancy-integration-by-parts}
Let $V$ be a compactly supported integral $n$-varifold with first variation $\delta V = - H_V \| V\| + \delta_s V$.
Let $U : \bR^{n+m} \to \bR$ be a locally semiconvex, $1$-Lipschitz map.
Then, for every non-negative compactly supported Lipschitz function $\varphi$, it holds that
\[
\int \varphi \Delta_V U \, d \|V\| \leq -\int \la \nabla^V \varphi, \nabla^V U \rg \, d \|V\| + \int \varphi \, d | \delta_s V |.
\]
Here, $\Delta_V U(x) := \textup{tr} \, D^2_V U(x)$ is the trace of the Alexandrov Hessian of $U$ restricted to a $C^2$ sheet representing $V$ at a good point.
\end{proposition}

The proof of Proposition~\ref{prop:fancy-integration-by-parts} relies on the following inequality for measures.
\begin{proposition}\label{prop:measure-zeta-q}
In the setting of Proposition~\ref{prop:fancy-integration-by-parts}, we write $\delta_s V = \zeta \, |\delta_s V|$, for $\zeta$ a $\bR^{n+m}$-valued function with $\delta_s V = \zeta \, |\delta_s V|$ and $|\zeta|=1$ holding $|\delta_sV|$-a.e.
We consider a signed Radon measure $\lambda$ and suppose that there exists a function $q \in L^{\infty}_{\textup{loc}}( |\delta_sV | ; \bR^{n+m})$ such that
\[
\int \psi \, d \lambda = - \int \la \nabla^V \psi, \nabla^V U \rg \, d\|V\| + \int \psi \la q, \zeta \rg \, d | \delta_s V|
\]
for every compactly supported $C^1$ function $\psi$ defined in an open neighborhood of $\on{spt} \|V\|$.
Moreover, suppose that the singular measure $\lambda_s$ in the Lebesgue decomposition of $\lambda$ with respect to $\|V\|$ satisfies $\lambda_s \geq 0$.
Then, $\lambda - \Delta_V U \, \|V\| \geq 0$ is a non-negative measure on $\bR^{n+m}$.
\end{proposition}
\begin{proof}
    First, we observe that since the claimed equality holds for all $C^1_c$ test functions, it is valid for compactly supported Lipschitz $\psi$.
    Indeed, by the standard approximation property of Lipschitz functions on rectifiable varifolds, as in~\cite{menne-sobolev-functions}*{Corollary 3.7} and~\cite{menne-2023}*{Theorem 6.20}, we can find a sequence of functions $\psi_{\ell} \in C_c^1( W)$ supported in a fixed bounded open set $\Omega$ and satisfying
    \[
    \psi_{\ell} \to \psi \quad \text{in } \; L^1_{\textup{loc}} ( |\delta_s V| + |\lambda|), \qquad \nabla^V \psi_{\ell} \to \nabla^V \psi \quad \text{in } \; L^1_{\textup{loc}}( \|V\|),
    \]
    in addition to $\sup_{\ell} |\nabla \psi_{\ell}| \leq C(\psi)$ and $\sup_{\ell} \|\psi_{\ell}\|_{L^{\infty}} \leq C(\psi)$.
%    For completeness, we recall why this approximation property applies here. Since $V$ is rectifiable, $\|V\|$-a.e.~ point lies on one of the countably many $C^2$, sheets covering $\textup{spt} \, \|V\|$, by Proposition~\ref{prop:menne}. On each sheet, the restriction of the Lipschitz function $\psi$ belongs to $W^{1,\infty}$, and its weak derivative agrees with the tangential gradient $\nabla^V \psi$. By Lusin's theorem and Whitney's extension theorem, the values and tangential first-order vectors of $\psi$ can be approximated on an exhausting sequence of compact subsets of the sheets by $C^1$ functions. Then, we can apply a diagonal argument over the countably many sheets to obtain the desired convergence of tangential gradients in $L^1_{\textup{loc}}( \|V\|)$. Multiplying by a fixed cutoff which is identically equal to $1$ on $\textup{spt} \, \varphi$ then preseves the compact support in $\Omega$.
    Combining this property with $q \in L^{\infty}_{\textup{loc}}( |\delta_s V| ; \bR^{m+n})$ and $|\nabla^V U| \leq 1$ holding $\|V\|$-a.e., we deduce that $\| \la q, \zeta \rg \|_{L^{\infty}} \leq \| q\|_{L^{\infty}} < + \infty$.
    Therefore, each of the above terms converges:
    \[
    \int \psi_{\ell} \, d \lambda \to \int \psi \, d \lambda, \qquad \int \la \nabla^V \psi_{\ell}, \nabla^V U \rg \, d \|V\| \to \int \la \nabla^V \psi, \nabla^V U \rg \, d \|V\|,
    \]
    and $\int \psi_{\ell} \la q, \zeta \rg \, d |\delta_s V| \to \int \psi \la q, \zeta \rg \, d |\delta_s V|$.
    Thus, we can take $\Omega$ to be a bounded open neighborhood of $\on{spt} \|V\|$ and consider $\psi \in \textup{Lip}_c(\Omega)$ in what follows.
    Moreover, we will work with the Lebesgue decomposition $\lambda = g \, \|V\| + \lambda_s$ of the measure $\lambda$, where $\lambda_s \perp \|V\|$ and $\lambda_s \geq 0$.
    Because the set of good points for $U$ has full $\|V\|$-measure, it suffices to prove the claimed inequality at a good point $x_0$ which is, in addition, a Lebesgue point for $g$ and satisfies $\frac{|\lambda_s| (B_r(x_0))}{r^n} \to 0$ and $\frac{|\delta_s V| (B_r(x_0))}{r^n} \to 0$ as $r \downarrow 0$.
    These additional properties also hold for $\|V\|$-a.e.~ good point, by the Besicovitch differentiation theorem and the properties $\lambda_s \perp \|V\|$ and $\delta_s V \perp \|V\|$.
    In what follows, we fix such a good point $x_0$ and choose a representing sheet $M_j$ through $x_0$, with $x_0 \in \cB_j$. 
    We write $\theta(x) := \Theta^n( \|V\|, x)$ and $\theta_0 := \Theta^n ( \|V\|, x_0)$, so after restricting $\|V\|$ to $\cB_j$, we have 
    \[
    d \|V\| = \theta \, d \cH^n, \qquad \frac{\cH^n ( M_j \cap B_r(x_0) \setminus \cB_j)}{r^n} \to 0, \qquad \frac{1}{r^n} \int_{M_j \cap B_r(x_0) \cap \cB_j} |\theta - \theta_0| \, d \cH^n \to 0.
    \]
    Because $U$ is locally semiconvex, we can find some small $\rho>0$ and some $C<\infty$ such that $U(x) + \frac{C}{2} |x - x_0|^2$ is convex on $B_{\rho}(x_0) \Subset \Omega$.
    Since $U|_{M_j}$ is differentiable at $x_0$, the tangential projection of any subdifferential vector $p \in \partial ( U + \tfrac{1}{2} C |x - x_0|^2 ) (x_0)$ onto $T_{x_0}M_j$ is uniquely determined by $\la p, \tau \rg = D \bigl( U|_{M_j} + \tfrac{1}{2}C |x - x_0|^2 \bigr) (x_0) [\tau]$, upon applying the sub-gradient inequality to $\gamma(t)$ and $\gamma(-t)$, for any $C^1$ curve $\gamma \subset M_j$ with $\gamma(0) = x_0$ and $\gamma'(0) = \tau$.
    For fixed $\delta \in (0, \frac{1}{100})$, we define
    \[
    s_{\delta}(x) := - \la p, x - x_0 \rg + \tfrac{1}{2} (C+\delta) |x - x_0|^2
    \]
    and let $\hat{U}_{\delta}(x) := U(x) - U(x_0) + s_{\delta}(x)$.
    By the convexity of $U + \frac{C}{2} |x - x_0|^2$, we have $\hat{U}_{\delta}(x) \geq \frac{\delta}{2} |x - x_0|^2$ for all $x \in B_{\rho}(x_0)$.
    By construction, it holds that $\hat{U}_{\delta}(x_0) = 0$ and $\nabla^{M_j} \hat{U}_{\delta}(x_0) = 0$.
    Therefore, $\hat{U}_{\delta}|_{M_j}$ has a strict quadratic minimum at $x_0$, which implies that
    \[
    A_{\delta} \geq \delta \, \textup{id}_{T_{x_0} M_j}, \qquad \text{where } \; A_{\delta} := D^2_{M_j} ( \hat{U}_{\delta}|_{M_j} )(x_0).
    \]
    We compute that $D s_{\delta}(x) = - p + (C+\delta)(x - x_0)$, hence also $\Delta_V s_{\delta} = \textup{tr}_{T_x V} \, D^2 s_{\delta} + \la H_V, D s_{\delta} \rg$ and
    \[
    D^2 s_{\delta} = (C+\delta) I, \qquad \Delta_V s_{\delta}(x) = n(C+\delta) + \la H_V(x), - p + (C+\delta)(x - x_0) \rg.
    \]
    Since $x_0$ is a $\|V\|$-Lebesgue point of $H_V$, this expression is approximately continuous at $x_0$, so $x_0$ is also a $\|V\|$-Lebesgue point of the function $\Delta_V s_{\delta}$ on the density-one sheet $\cB_j \subset M_j$.
    The matrix $A_{\delta}$ is positive-definite, and $\textup{tr} \, A_{\delta} = \Delta_V U(x_0) + \Delta_V s_{\delta}(x_0)$ because $M_j$ represents $V$ at $x_0$, while $H_{M_j}(x_0) = H_V(x_0)$ at a good point.
    Moreover, 
    \[
    \int_{\Omega} \psi \Delta_V s_{\delta} \, d \|V \| = - \int_{\Omega} \la \nabla^V \psi , \nabla^V s_{\delta} \rg \, d \|V\| + \int_{\Omega} \psi \la D s_{\delta}, \zeta \rg \, d |\delta_s V|
    \]
    by the usual integration by parts on the varifold $V$, since $s_{\delta}$ is smooth, cf.~\cite{menne-2023}*{Example 17.12}.
    We consider the signed Radon measure $\tilde{\lambda}_{\delta} := \lambda + ( \Delta_V s_{\delta}) \|V\|$, so our previous extension implies that
\begin{equation}\label{eqn:signed-radon-measure-q}
    \int_{\Omega} \psi \, d \tilde{\lambda}_{\delta} = - \int_{\Omega} \la \nabla^V \psi , \nabla^V \hat{U}_{\delta} \rg \, d \|V\| + \int_{\Omega} \psi \la q + D s_{\delta}, \zeta \rg \, d |\delta_s V|
\end{equation}
    for every Lipschitz $\psi \in \textup{Lip}_c(\Omega)$.
    The measure $\tilde{\lambda}_{\delta}$ has absolute continuous density with respect to $\|V\|$ equal to $g + \Delta_V s_{\delta}$.
    Moreover, to prove that $\lambda - \Delta_V U \, \|V\| \geq 0$ is a non-negative measure, it suffices to show that 
    $\Delta_V U(x_0) \leq \frac{d \lambda}{d \|V\|}(x_0) = g(x_0)$ for $\|V\|$-a.e.~ good point as in Definition~\ref{def:alexandrov-point}.
    This reduction is valid because the singular measure in the Lebesgue decomposition satisfies $\lambda_s \geq 0$.
    In terms of the measure $\tilde{\lambda}_{\delta}$, this inequality becomes equivalent to 
    \begin{equation}\label{eqn:tr-a-delta-bound}
        g(x_0) + \Delta_V s_{\delta}(x_0) \geq \textup{tr} \, D^2_{M_j} ( \hat{U}_{\delta}|_{M_j}) (x_0) =: \textup{tr}\, A_{\delta}.
    \end{equation}
    We consider a non-zero cutoff function $\eta \in C_c^{\infty}([0,\infty))$ such that $\eta, - \eta' \geq 0$ and $\on{spt} \eta \subset [0,1]$, with $\eta \equiv 1$ on $[ 0, \frac{1}{2}]$.
    For small $r < \frac{1}{20} \delta^2 \rho$, we define 
    \[
    \psi_r(x) := \eta( r^{-2} \hat{U}_{\delta}(x)) \, \eta (  4\rho^{-1} |x - x_0|)
    \]
    so the property $\hat{U}_{\delta}(x) \geq \frac{\delta}{2} |x - x_0|^2$ implies that $\psi_r(x) = 0$ for $x \not\in B_{2 \delta^{-1} r}(x_0)$.
    Moreover, $\psi_r(x) = 0$ for $x \not\in B_{ \rho/4 }(x_0)$.
    In particular, $\on{spt} \psi_r \subset B_{\rho/4}(x_0) \Subset B_{\rho}(x_0)$ is Lipschitz and compactly supported in $B_{\rho}(x_0)$, for all sufficiently small $r$, so it is an admissible test function in~\eqref{eqn:signed-radon-measure-q}.
    In addition, for $x \in B_{2 \delta^{-1}r}(x_0)$, we have $4 \rho^{-1} |x-x_0| \leq \frac{8 r}{\delta \rho} < \frac{8 \delta^2 \rho}{20 \delta \rho} < \frac{1}{2}$, so $\eta ( 4 \rho^{-1} |x - x_0|) \equiv 1$ and $\psi_r = \eta ( \frac{\hat{U}_{\delta}}{r^2})$ on $\on{spt} \psi_r$.
    Using the chain rule, we compute $\|V\|$-a.e.~ that $\nabla^V \psi_r = r^{-2} \eta' ( \frac{\hat{U}_{\delta}}{r^2}) \nabla^V \hat{U}_{\delta}$, hence
    \begin{equation}\label{eqn:psir-integral}
    \int_{\Omega} \psi_r \, d \tilde{\lambda}_{\delta} = \int_{\Omega} \psi_r \la q + D s_{\delta}, \zeta \rg \, d |\delta_s V| - \int_{\Omega} \eta' \Bigl( \frac{\hat{U}_{\delta}}{r^2} \Bigr) \frac{|\nabla^V \hat{U}_{\delta}|^2}{r^2} \, d \|V\|.
    \end{equation}
    Moreover, since $q \in L^{\infty}_{\textup{loc}}( |\delta_s V|)$ and $\on{spt} \psi_r \subset B_{2 \delta^{-1} r}(x_0)$, while $D s_{\delta}$ is smooth, we can bound
    \[
    \left| \int_{\Omega} \psi_r \la q + D s_{\delta} , \zeta \rg \, d |\delta_s V| \right| \leq C(\delta, p, q) \, |\delta_s V| (B_{2 \delta^{-1} r} (x_0)) = o (r^n)
    \]
    because $x_0$ is a good point.
    To compute the small-$r$ asymptotics of the left-hand side, we choose a $C^2$ diffeomorphism $\Phi: B^n_R \to M_j$ onto the sheet $M_j$, with $\Phi(0) = x_0$ and $D \Phi(0): \bR^n \xrightarrow{\simeq} T_{x_0} M_j$.
    Because $D \Phi(0)$ is an isometry, after shrinking the chart we can assume that $\Phi$ satisfies the bi-Lipschitz estimate $c|y| \leq |\Phi(y) - x_0| \leq C |y|$ for every small $y$.
    We will use the change of variables $y = rz$ throughout, whereby $\on{spt} \psi_r \subset B_{2 \delta^{-1} r}(x_0)$ shows that
    \[
    \psi_r ( \Phi(rz)) \neq 0 \implies  cr |z| \leq  |\Phi(rz) - x_0| \leq 2 \delta^{-1} r \implies |z| \leq 2 ( c \delta)^{-1}. 
    \]
    Consequently, any non-negative contribution in the integral equality~\eqref{eqn:psir-integral} comes from the fixed compact domain $B_{2(c\delta)^{-1}}(0)$ in the $z$-variable.
    
    Since $x_0$ is an Alexandrov point for $\hat{U}_{\delta}|_{M_j}$ and $\hat{U}_{\delta}|_{M_j}$ is differentiable at $x_0$, we can express
    \begin{equation}\label{eqn:u-hat-a.e.z}
\begin{split}
    \hat{U}_{\delta} ( \Phi(y)) &= \tfrac{1}{2} A_{\delta} [y,y] + o ( |y|^2), \qquad \nabla^{M_j} \hat{U}_{\delta} ( \Phi(y)) = D \Phi(0) (A_{\delta} y) + o ( |y|) \quad \text{as } \; y \to 0.
\end{split}        
    \end{equation}
    Indeed, the function $\hat{U}_{\delta} \circ \Phi$ is semiconvex in a neighborhood of the origin, so $v(y) := \hat{U}_{\delta} \circ \Phi + \frac{\kappa}{2} |y|^2$ is convex, satisfies $v(y) = \frac{1}{2} (A_{\delta} + \kappa I) [y,y] + o (|y|^2)$, and is differentiable on a set of full measure, by Alexandrov's theorem~\cite{villani}*{Theorems 14.1 and 14.25}.
    For a differentiability point $y \neq 0$, a unit vector $e$, and any $t \in (0,\frac{1}{2})$, the convexity of $v$ implies that
    \[
    \frac{v(y) - v(y - t |y|e)}{t |y|} \leq \la D v(y), e \rg \leq \frac{v(y + t |y| e) - v(y)}{t |y|}.
    \]
    Thus, $v(y) = \frac{1}{2} (A_{\delta} + \kappa I)[y,y] + o (|y|^2)$ implies $\limsup_{y \to 0} \frac{|D v(y) - (A_{\delta} + \kappa I) y|}{|y|} \leq C(A_{\delta},\kappa) t$, and sending $t \downarrow 0$ shows that $Dv(z) = (A_{\delta} + \kappa I) y + o (|y|)$.
    This proves the expansion~\eqref{eqn:u-hat-a.e.z} as $y \to 0$.
    Moreover,
    \[
    \left| \frac{1}{r^n} \int_{\Omega} \psi_r \, d \|V\| - \frac{1}{r^n} \int_{\cB_j} \psi_r \, d \|V\| \right| \leq \frac{\| \psi_r \|_{L^{\infty}}}{r^n} \|V\| ( \textup{spt} \, \psi_r \setminus \cB_j) \leq \frac{C}{r^n} \|V\| (B_{2 \delta^{-1}r} (x_0) \setminus \cB_j)
    \]
    due to $\textup{spt} \, \psi_r \setminus \cB_j \subset B_{2 \delta^{-1}r}(x_0) \setminus \cB_j$.
    Using the density properties of $\cB_j$ and $\theta$ near $x_0$, as in Lemma~\ref{lemma:v-ae-q-gj}, $\|V\|(B_{2 \delta^{-1}r}(x_0) \setminus \cB_j) = o(r^n)$ as $r \to 0$.
    The same argument applies to $\tilde{\lambda}_{\delta}$, since the singular part has a lower-order contribution $|\lambda_s| (B_{2 \delta^{-1} r}(x_0)) = o (r^n)$ at a good point.
    Therefore, as $r \to 0$,
    \[
    \left| \frac{1}{r^n} \int_{\Omega} \psi_r \, d \|V\| - \frac{1}{r^n} \int_{\cB_j} \psi_r \, d \|V\| \right| \to 0, \qquad \left| \frac{1}{r^n} \int_{\Omega} \psi_r \, d \tilde{\lambda}_{\delta} - \frac{1}{r^n} \int_{\cB_j} \psi_r \, d \tilde{\lambda}_{\delta} \right| \to 0
    \]
    Next, in the image of the compact $z$-domain $B_{4 ( c \delta)^{-1}}(0)$ inside the Borel set $\cB_j$, we use the expression~\eqref{eqn:u-hat-a.e.z} to write $r^{-2} \hat{U}_{\delta} (\Phi(rz)) \to \tfrac{1}{2} A_{\delta} [z,z]$ for $z \in B_{4 ( c \delta)^{-1}}(0)$, uniformly as $r \downarrow 0$.
    Moreover, the Jacobian of the map $\Phi$ satisfies $J_{\Phi}(r,z) \to 1$.
    Because $x_0$ is a $\|V\|$-Lebesgue point of $g$ and $\Delta_V s_{\delta}$ is $\|V\|$-approximately continuous at $x_0$, while $|\lambda_s| (B_{2 \delta^{-1} r}(x_0)) = o(r^n)$, we obtain
    \begin{align*}
    \frac{1}{r^n} \int_{\cB_j} \psi_r \, d \|V\| &\to \Theta^n( \|V\|,x_0)\int_{\bR^n} \eta ( \tfrac{1}{2} A_{\delta}[z,z] ) \, dz, \\
    \frac{1}{r^n} \int_{\cB_j} \psi_r \, d \tilde{\lambda}_{\delta} &\to \Theta^n (\|V\|,x_0) \, ( g(x_0) + \Delta_V s_{\delta}(x_0)) \int_{\bR^n} \eta (\tfrac{1}{2} A_{\delta}[z,z]) \,d z.
    \end{align*}
    It remains to examine the third term in the equality~\eqref{eqn:psir-integral}.
    Since $- \eta' \geq 0$, the integrand is everywhere non-negative, so we may restrict to the single sheet $\cB_j$ and estimate
    \[
    \frac{1}{r^n} \int_{\Omega} - \eta' \Bigl( \frac{\hat{U}_{\delta}}{r^2} \Bigr) \frac{|\nabla^V \hat{U}_{\delta}|^2}{r^2} \, d \|V\| \geq \frac{1}{r^n} \int_{\cB_j} -\eta' \Bigl( \frac{\hat{U}_{\delta}}{r^2} \Bigr) \frac{|\nabla^{M_j} \hat{U}_{\delta}|^2}{r^2} \, d \|V\|.
    \]
    After pulling back by the map $\Phi$ and changing variables to $y =rz$, the latter integral becomes
    \[
    I_r = \int_{\bR^n} \mathbf{1}_{\cB_j} ( \Phi(rz)) \theta ( \Phi(rz)) J_{\Phi}(rz) \Bigl[ - \eta' \Bigl( \frac{\hat{U}_{\delta} ( \Phi(rz))}{r^2} \Bigr) \frac{|\nabla^{M_j} \hat{U}_{\delta} ( \Phi(rz))|^2}{r^2} \Bigr] \, dz
    \]
    since $d \|V\| = \theta \, d \cH^n$ on $\cB_j \subset M_j$.
    By our previous observations, this expression is supported in the compact region $B_{4 ( c \delta)^{-1}}(0) \Subset \bR^n$ in the $z$-variable, so we may again use the $L^1$-density properties of $\cB_j$ and $\theta$ near $x_0$, along with $J_{\Phi}(r,z) \to 1$, in order to obtain a sequence of radii $r_k \downarrow 0$ with
    \begin{align*}
    & \mathbf{1}_{\cB_j} ( \Phi(r_k z)) \, \theta ( \Phi(r_k z)) \, J_{\Phi}(r_k z) \to \theta_0 \quad \text{in } \; L^1( B_{4 ( c \delta)^{-1}}(0)), \\
    & \quad \implies \quad \mathbf{1}_{\cB_j} ( \Phi(r_{k'} z)) \, \theta ( \Phi(r_{k'} z)) \, J_{\Phi}(r_{k'} z) \to \theta_0 \quad \text{for a.e.~ } z \in B_{4 (c \delta)^{-1}}(0)
    \end{align*}
    after passing to a further subsequence.
    This property follows from the fact that $0$ is a density point of $\Phi^{-1}(\cB_j)$ and a Lebesgue point of $\theta \circ \Phi$, while $J_{\Phi}(rz) \to 1$ locally uniformly.
    Finally, using the uniform convergence $\frac{1}{r^2} \hat{U}_{\delta} (\Phi(rz)) \to \frac{1}{2} A_{\delta} [z,z]$ and the properties 
\[
r_{\ell}^{-2} \hat{U}_{\delta} ( \Phi(r_{\ell}z)) \to \tfrac{1}{2} A_{\delta}[z,z], \qquad r_{\ell}^{-1} \nabla^{M_j} \hat{U}_{\delta}( \Phi(r_{\ell}z)) \to A_{\delta} z
\]
for a.e.~ $z$, due to~\eqref{eqn:u-hat-a.e.z}, we see that the whole non-negative rescaled integrand converges a.e.~ to $\theta_0 \bigl[ - \eta' \bigl( \tfrac{1}{2} A_{\delta}[z,z]  \bigr) |A_{\delta} z|^2 \bigr]$.
Therefore, choosing a sequence $r_k \downarrow 0$ that attains $\liminf_{r \downarrow 0} I_r$, we obtain
\begin{align*}
    \liminf_{r \downarrow 0} I_r &= \lim_{\ell \to \infty} I_{r_{\ell}} \geq \int_{\bR^n} \liminf_{\ell \to \infty} \mathbf{1}_{\cB_j} ( \Phi(rz)) \theta ( \Phi(rz)) J_{\Phi}(rz) \Bigl[ - \eta' \Bigl( \frac{\hat{U}_{\delta} ( \Phi(rz))}{r^2} \Bigr) \frac{|\nabla^{M_j} \hat{U}_{\delta} ( \Phi(rz))|^2}{r^2} \Bigr] \, dz \\
    &= \theta_0 \int_{\bR^n} - \eta' \bigl( \tfrac{1}{2} A_{\delta}[z,z] \bigr) |A_{\delta} z|^2 \, dz \\
    &= \theta_0 \, \textup{tr}\, A_{\delta} \, \int_{\bR^n} \eta \bigl( \tfrac{1}{2} A_{\delta}[z,z] \bigr) \, dz.
\end{align*}
by Fatou's lemma.
In the last step, we used $\nabla ( \frac{1}{2} A_{\delta}[z,z]) = A_{\delta} z$ and $\textup{div}(A_{\delta} z) = \textup{tr}\, A_{\delta}$ and integrated by parts, because the vector field is compactly supported.
    Returning to the identity~\eqref{eqn:psir-integral} for $\psi_r$, dividing by $r^n$, and taking the $\liminf_{r \downarrow 0}$ of both expressions, we conclude that
    \[
    \theta_0 ( g(x_0) + \Delta_V s_{\delta}(x_0)) \int_{\bR^n} \eta \bigl( \tfrac{1}{2} A_{\delta}[z,z] \bigr) \, dz \geq \theta_0 \, \on{tr} A_{\delta} \int_{\bR^n} \eta ( \tfrac{1}{2} A_{\delta}[z,z] ) \, dz.
    \]
    Eliminating the non-negative factors $\theta_0 > 0$ and $\int_{\bR^n} \eta ( \frac{1}{2} A_{\delta}[z,z]) \, dz > 0$ proves~\eqref{eqn:tr-a-delta-bound}.
    Thus, $\lambda - \Delta_V U \, \|V\| \geq 0$ is a non-negative Radon measure, as desired.
\end{proof}

We now combine these steps to prove Proposition~\ref{prop:fancy-integration-by-parts}.
\begin{proof}[Proof of Proposition~\ref{prop:fancy-integration-by-parts}]
We take $U_{\ve} := U \ast \chi_{\ve}$ to be a sequence of mollifications of $U$ and let $\psi$ be any compactly supported $C^1$ function defined in a bounded open neighborhood $\Omega$ of $\textup{spt} \, \|V\|$.
By~\cite{payne-redaelli}*{Lemma 2.4}, for $\ve>0$ sufficiently small, the function $U_{\ve}$ is locally semiconvex, $1$-Lipschitz, $U_{\ve} \to U$ locally uniformly as $\ve \downarrow 0$, and $D^2 U_{\ve} \geq - C \, \textup{id}$ on $\Omega$ for a constant $C$ independent of $\ve$.
On a $C^2$ sheet $M_j$ representing $V$ in the sense of Proposition~\ref{prop:menne}, we may therefore bound
\[
\Delta_V U_{\ve} = \textup{tr}_V \, D^2_V U_{\ve} = \Delta_{M_j}(U_{\ve}|_{M_j}) = \textup{tr}_{T_x M_j} \, D^2 U_{\ve} + \la H_V , \nabla U_{\ve} \rg \geq - nC - |H_V|.
\]
This property holds $\|V\|$-a.e.~ on $\Omega$, and hence $\lambda_{\ve} := (\Delta_V U_{\ve} + 2nC + |H_V|) \, \|V\| \mres \Omega$ is a positive Radon measure for every $\ve > 0$ sufficiently small, since $H_V \in L^1_{\textup{loc}}(\|V\|)$ due to the locally bounded first variation.
Writing $\delta V = - H_V \|V\| + \delta_s V$ with $\delta_s V = \zeta \, |\delta_s V|$ and using the first variation formula for the vector field $X = \psi D U_{\ve} \in C^1_c ( \Omega ; \bR^{n+m})$, we obtain
\begin{equation}\label{eqn:phi-U-eps-IBP}
\begin{split}
 \int \psi \Delta_V U_{\ve} \, d\|V\| &= -\int \la \nabla^V \psi, \nabla^V U_{\ve} \rg \, d\|V\| + \int \psi \la DU_{\ve} , \zeta \rg \, d |\delta_s V| \\
 &\leq \int |\nabla^V \psi| \, d\|V\|+ \int \psi \, d |\delta_s V|
\end{split}
\end{equation}
because $U_{\ve}$ is $1$-Lipschitz, so $\la DU_{\ve}, \zeta \rg \leq 1$ and $-\la \nabla^V \psi, \nabla^V U_{\ve} \rg \leq |\nabla^V \psi|$.
For every compact set $K \Subset \Omega$, we can use this bound with a non-negative cutoff function $\eta \in C_c^1(W)$ satisfying $\eta \equiv 1$ on $K$ to see that the masses $\lambda_{\ve}(K)$ are uniformly bounded.
Thus, we can extract a sequence of $\ve_k \downarrow 0$ and a positive Radon measure $\tilde{\lambda}$ with $\lambda_{\ve_k} \xrightharpoonup{*} \tilde{\lambda}$ as Radon measures on $W$.
The functions $U_{\ve_k}$ are $1$-Lipschitz, so after passing to a further subsequence, we can furthermore arrange that $DU_{\ve_k} \xrightharpoonup{*} q$ in $L^{\infty}_{\textup{loc}}( |\delta_s V| ; \bR^{n+m})$, for some $q$ with $|q| \leq 1$, $|\delta_sV|$-a.e.
We now define the signed Radon measure $\lambda$ by
\[
\lambda := \tilde{\lambda} - (2n C + |H_V|) \, \|V\|, \qquad \Delta_V U_{\ve_k} \, \|V \| \xrightharpoonup{*} \lambda.
\]
On every $C^2$ sheet $M_j$, the restrictions $U_{\ve_k}|_{M_j}$ are locally uniformly semiconvex and converge locally uniformly to $U|_{M_j}$.
In a $C^2$ coordinate chart, after adding a fixed quadratic function, the restrictions are convex.
Hence, the convergence theorem for gradients of convex functions, cf.~\cite{rockafellar}*{Theorem 24.5}, gives $\nabla^V U_{\ve_k} \to \nabla^V U$ $\|V\|$-a.e.~, and therefore strongly in $L^1_{\textup{loc}}(\|V\|)$ by dominated convergence.
Therefore, taking limits as $k \to \infty$ in the first line of equation~\eqref{eqn:phi-U-eps-IBP}, we obtain
\[
\int \psi \, d \lambda = - \int \la \nabla^V \psi, \nabla^V U \rg \, d\|V\| + \int \psi \la q, \zeta \rg \, d |\delta_s V|.
\]
Because $\tilde{\lambda}$ is a positive Radon measure, its Lebesgue decomposition with respect to $\|V\|$ has singular measure satisfying $\tilde{\lambda}_s \geq 0$.
Also, the choice of $\psi$ was arbitrary, so Proposition~\ref{prop:measure-zeta-q} applies to show that $\lambda - \Delta_V U \, \|V\| \geq 0$ is a non-negative Radon measure.
Thus, we can apply Fatou's lemma to see that any non-negative function $\varphi \in C_c^1(\Omega)$ satisfies
\[
\int \varphi \Delta_V U \, d\|V\| \leq \int \varphi \, d \lambda \leq \limsup_{k \to \infty} \int \varphi \Delta_V U_{\ve_k} \, d\|V\| \leq - \int \la \nabla^V \varphi, \nabla^V U \rg \, d \|V\| + \int \varphi \, d |\delta_s V|.
\]
Finally, the general inequality, where $\varphi \in \textup{Lip}_c(\Omega)$ is Lipschitz instead of $C^1_c(\Omega)$, follows from the approximation argument used in the beginning of Proposition~\ref{prop:measure-zeta-q}.
This completes the proof.
\end{proof}

\section{An optimal transportation scheme for varifolds}\label{section:optimal-transportation-varifolds}

In what follows, we consider integers $n \geq 2$ and $m \geq 2$ and assume that $V$ is a compactly supported $n$-varifold in $\bR^{n+m}$, so $\| V \| ( \bR^{m+n} )< \infty$ and the density function satisfies $\Theta^n ( \|V\|, x) \in (0,\infty)$ for $\|V\|$-a.e.~ $x$.
As in~\eqref{eqn:HV-decomposition}, we decompose the vector-valued Radon measure $\delta V$ as $\delta V = - H_V \|V\| + \delta_s V$ with respect to $\|V\|$.
We will prove the following result.
\begin{theorem}\label{thm:optimal-MSS}
Let $V$ be a compactly supported integral $n$-varifold with locally bounded first variation in $\bR^{n+m}$.
Every Lipschitz function $f$ defined in a neighborhood of $\on{spt} \|V\|$ satisfies
    \begin{align*}
    & n \Bigl( \frac{(n+m) \, |B^{n+m}|}{m |B^m|} \Bigr)^{\frac{1}{n}} \Bigl( \int |f|^{\frac{n}{n-1}} \, d \|V\| \Bigr)^{- \frac{1}{n}} \int \Theta^n ( \|V\|, x)^{\frac{1}{n}} |f|^{\frac{n}{n-1}} \, d \|V\|  \\
    &\leq \int \sqrt{|\nabla^V f|^2 + |H_V|^2 f^2}  \, d \|V\| + \int |f| \, d |\delta_s V|.
    \end{align*}
In particular, expressing the total first variation measure as $\| \delta V\| = |H_V| \, \|V\| + |\delta_s V|$, we have
    \[
n \Bigl( \frac{(n+m) |B^{n+m}|}{m |B^m|} \Bigr)^{\frac{1}{n}} \, \Bigl( \int |f|^{\frac{n}{n-1}} \, d\|V\| \Bigr)^{\frac{n-1}{n}} \leq \int |\nabla^V f| \, d\|V\| + \int |f| \, d \| \delta V \|.
\]
\end{theorem}
We will establish Theorem~\ref{thm:optimal-MSS} using an optimal transportation estimate.
Let $\rho : [0, \infty) \to (0,\infty)$ be a continuous function with $\int_{\bar{B}^{m+n}} \rho ( |\xi|^2) \, d \xi = 1$, with $\bar{B}^{m+n} = \{ \xi \in \bR^{m+n} : |\xi| \leq 1 \}$ the closed unit ball in $\bR^{n+m}$, and
\begin{equation}\label{eqn:alpha-rho-constant}
   \alpha_{\rho} := \sup_{z \in \bR^n} \int_{ \{ y \in \bR^m : |z|^2 + |y|^2 \leq 1 \} } \rho ( |z|^2 + |y|^2) \, dy. 
\end{equation}
We will prove the following result.
\begin{proposition}\label{prop:weighted-MSS}
Let $\varphi$ and $f$ be Lipschitz functions defined in a neighborhood of $\on{spt} \|V\|$, with $\varphi$ a non-negative compactly supported function.
Then, it holds that
\begin{align*}
    &n \alpha_{\rho}^{- \frac{1}{n}} \Bigl( \int |f|^{\frac{n}{n-1}} \, d\|V\| \Bigr)^{- \frac{1}{n}} \int \Theta^n ( \|V\|, x)^{\frac{1}{n}} |f|^{\frac{n}{n-1}} \varphi(x) \, d \|V\|(x) \\
    &\leq \int \varphi \sqrt{ |\nabla^V f|^2 + |H_V|^2 f^2 } \, d \|V\| + \int |f| \, \varphi \, d |\delta_s V| + \int |f| \, |\nabla^V \varphi| \, d \|V\|.
\end{align*}
\end{proposition}
To prove Proposition~\ref{prop:weighted-MSS}, we will utilize an optimal transportation scheme.
We define a Borel measure $\nu$ on the unit ball $\bar{B}^{m+n}$ by $\nu(G) := \int_G \rho( |\xi|^2) \, d \xi$ for every Borel set $G \subset \bar{B}^{m+n}$.
Equivalently, $d \nu(\xi) = \rho ( |\xi|^2) \, d \xi$, and $\nu$ is a probability measure.
We will first prove Proposition~\ref{prop:weighted-MSS} when the function $f$ is everywhere positive.
In what follows, we define a probability measure $\mu_f$ on $\bR^{m+n}$ by
\[
d \mu_f = \frac{f(x)^{\frac{n}{n-1}}}{ \int f^{\frac{n}{n-1}} \, d \|V\| } d \|V\|(x),
\]
where $\int f^{\frac{n}{n-1}} \, d \|V\| < \infty$ because $f$ is Lipschitz and $V$ is compactly supported.
After scaling $f$ by a constant positive multiple, we can assume for simplicity that $\int f^{\frac{n}{n-1}} \, d \|V\| = 1$, so $\mu_f = f^{\frac{n}{n-1}} \|V\|$.
We extend our notion of good points $x_0$ from Definition~\ref{def:alexandrov-point} to assume that $f$ is differentiable at $x_0$ and $\nabla^V f(x_0)$ is well-defined.
By Rademacher's theorem, the resulting set has full $\|V\|$-measure.

We denote by $\Pi(\mu_f, \nu)$ the set of all transport plans from $\mu$ to $\nu$, namely probability measures 
\[
\Pi( \mu_f, \nu) := \{ \gamma \in \cP ( \on{spt} \|V\| \times \bar{B}^{m+n}) : \; (\pi_{\on{spt} \|V\|})_{\#} \gamma = \mu_f , \; (\pi_{\bar{B}^{m+n}})_{\#} \gamma = \nu \}.
\]
Let $\cJ$ denote the set of pairs $(u,h)$ such that $u \in L^1( \| V\|)$ is an integrable function on $\on{spt} \|V\|$, $h$ is an integrable function on $\bar{B}^{m+n}$, and 
\begin{equation}\label{eqn:u-h-property}
u(x) - h(\xi) - \la x, \xi \rg \geq 0
\end{equation}
for $\mu_f \otimes \nu$-a.e.~ point $(x,\xi)$ with $x \in \on{spt} \|V\|$ and $\xi \in \bar{B}^{m+n}$.
Arguing as in~\cite{brendle-eichmair}*{\S 2} and following~\cite{villani}*{\S 10 and \S 14}, we can find a pair $(u,h) \in \cJ$ which maximizes the functional
\[
\int_{\bar{B}^{m+n}} h \, d \nu - \int_{\textup{spt} \, \|V\|} u \, d \mu_f,
\]
and such that $h$ is Lipschitz continuous.
After replacing $u$ by its $c$-transform as in~\cite{villani}*{Theorem 5.10}, we may assume that $u$ is the restriction to $\on{spt} \| V\|$ of the ambient convex function
\begin{equation}\label{eqn:u-property}
    U(x) = \sup_{\xi \in \bar{B}^{n+m}} ( h(\xi) + \la x, \xi \rg)
\end{equation}
which is convex and $1$-Lipschitz on $\bR^{m+n}$.
Moreover, we can find an optimal transport plan $\gamma \in \Pi(\mu_f, \nu)$ for the cost $c(x,\xi) = - \la x, \xi \rg$, for which Kantorovich duality implies 
\[
U(x) - h(\xi) - \la x, \xi \rg = 0 \qquad \text{for } \; \gamma\text{-a.e.~  } \; (x,\xi). 
\]
Equivalently, the measure $\gamma$ is concentrated on the contact set
\begin{equation}\label{eqn:contact-set-Gamma-U}
\Gamma_U := \{ (x,\xi) \in \on{spt} \|V\| \times \bar{B}^{n+m} : U(x) - h(\xi) - \la x, \xi \rg = 0 \}.
\end{equation}
Since $V$ is rectifiable, Rademacher's theorem implies that $U$ has a tangential derivative $\nabla^VU$ at $\| V\|$-a.e.~ point, and $|\nabla^VU| \leq 1$.
Moreover, the equality~\eqref{eqn:u-property} implies that $u = U|_{\on{spt} \|V\|}$ satisfies the property~\eqref{eqn:u-h-property} for every $(x,\xi) \in \on{spt} \|V\| \times \bar{B}^{m+n}$.
Since $U$ restricts to a semiconvex function on each $C^2$ sheet $M_j$ representing $V$, Proposition~\ref{prop:menne} applies to the varifold $V$ and the map $U$.
\begin{lemma}\label{lemma:E-Sigma-perturb}
    Let $E$ be a Borel subset of $\textup{spt} \, \|V\|$ and let $G$ be a Borel subset of $\bar{B}^{m+n}$ such that $U(x) - h(\xi) - \la x, \xi \rg > 0$ for all $x \in E$ and $\xi \in \bar{B}^{n+m} \setminus G$.
    Then, $\mu_f(E) \leq \nu(G)$.
\end{lemma}
\begin{proof}
The optimal transport plan $\gamma \in \Pi(\mu_f,\nu)$ is concentrated on the contact set $\Gamma_U$ of~\eqref{eqn:contact-set-Gamma-U}, so $\gamma(\Gamma_U) = 1$.
Our assumptions imply that every contact point over $E$ has its second coordinate in $G$, namely $\Gamma_U \cap ( E \times \bar{B}^{n+m}) \subset E \times G \subset \on{spt} \|V\| \times G$.
We therefore obtain
\begin{align*}
\mu_f(E) &= \gamma( E \times \bar{B}^{n+m}) = \gamma ( \Gamma_U \cap (E \times \bar{B}^{n+m})) \leq \gamma ( \on{spt} \|V\| \times G) \leq \nu(G).
\end{align*}
This proves our assertion.
\end{proof}

In what follows, we fix a good point $x_0 \in \textup{spt} \, \|V\|$, as in Definition~\ref{def:alexandrov-point}, and a $C^2$ sheet $M_j$ representing $V$ to second order near $x_0$, with $x_0 \in \cB_j$.
Since $M_j$ is $C^2$, there is a constant $K_j < \infty$ and a small $\rho>0$ such that for every $x,z \in M_j \cap B_{\rho}(x_0)$ and every $\eta \in T_x^{\perp} M_j$ with $|\eta| \leq1$, we have $\la z-x, \eta \rg \geq - K_j d_j(x,z)^2$, where $d_j$ is the distance function on $M_j$.
We define the sheet-wise quadratic subdifferential of $U|_{M_j}$ at $x \in M_j \cap B_{\rho}(x_0)$ by
\[
\partial_{M_j} U(x) := \{ v \in T_x M_j : U(z) - U(x) - \la z-x, v \rg \geq - K_j d_j(x,z)^2, \; \text{for all } \; z \in M_j \cap B_{\rho} (x_0) \}.
\]
Given a vector $\xi \in \bR^{n+m}$ and a good point $x \in \textup{spt} \, \|V\|$, we denote by $\xi^{\top} := \pi_{T_x V} \xi$ the projection of $\xi$ to the approximate tangent space of the varifold $V$.
\begin{lemma}\label{lemma:BE-lemma-4}
Consider a point $x \in M_j \cap B_{\rho}(x_0)$ and some $\xi \in \bar{B}^{n+m}$.
If $U(x) - h(\xi) - \la x, \xi \rg = 0$, then $\pi_{T_x M_j} \xi \in \partial_{M_j} U(x)$.
In particular, if $x \in \textup{spt} \|V\| \cap M_j \cap B_{\rho}(x_0)$ is a good point, then $\xi^{\top} \in \partial_{M_j} U(x)$.
\end{lemma}
\begin{proof}
    The proof is identical to~\cite{brendle-eichmair}*{Lemma 4}.
\end{proof}
Let $\hat{U}$ be a smooth function on $M_j$ with $|U(x) - \hat{U}(x)| \leq o ( d_j(x, x_0)^2)$ as $x \to x_0$.
For each $r \in (0, \rho)$, we denote by $\hat{\omega}$ the smallest non-negative real number $\omega$ with the property that
\[
| \zeta - \nabla^{M_j} \hat{U}(x) | \leq \omega, \qquad \text{for all} \quad x \in M_j \cap \bigl\{ d_j(x,x_0) \leq \sqrt{n} r \bigr\} \quad \text{and} \quad \zeta \in \partial_{M_j} U(x).
\]
We denote by $\hat{\delta}(r)$ the smallest non-negative real number $\delta$ with the property that
\[
D^2_{M_j} \hat{U}(x) - \la \mathrm{I\!I}_{M_j}(x), \pi_{T_x^{\perp} M_j} \xi \rg \geq - \delta \, g_{M_j}, \qquad \forall \; (x,\xi) \in \Gamma_U, \quad x \in M_j \cap \{ d_j(x,x_0) \leq \sqrt{n} r \}.
\]
\begin{lemma}\label{lemma:BE-5-and-6}
    The functions $\hat{\omega}, \hat{\delta} : (0,\rho) \to [0,\infty)$ are increasing and $\lim_{r \downarrow 0} \frac{\hat{\omega}(r)}{r} = \lim_{r \downarrow 0} \hat{\delta}(r) = 0$.
\end{lemma}
\begin{proof}
    These properties follow from~\cite{brendle-eichmair}*{Lemma 5 and Lemma 6} applied to the sheet $M_j$, by using the Alexandrov expansion $U - \hat{U} = o ( d_j(-,x_0)^2)$ and the second-order supporting inequality at contact points.
    The latter property applies here due to the continuity of $\mathrm{I\!I}_{M_j}$ and $D^2_{M_j} \hat{U}$.
\end{proof}

In what follows, we let $E_r := Q_r \cap \cB_j \cap \textup{spt} \, \|V\|$, so $\|V\|(E_r) = \theta(x_0) r^n + o(r^n)$ by Lemma~\ref{lemma:v-ae-q-gj}.
We define a subset $A_r \subset T^{\perp} M_j|_{Q_r}$ of the normal bundle of $Q_r$ inside $M_j$ by
\begin{equation}\label{eqn:ar-define}
\begin{split}
A_r := \{ (z,\eta): z \in Q_r, \eta \in T^{\perp}_z M_j:  \; & |\nabla^{M_j} \hat{U}(z)|^2 + |\eta|^2 \leq (1 + \hat{\omega}(r))^2, \\
&D^2_{M_j} \hat{U}(z) - \la \mathrm{I\!I}_{M_j}(z), \eta\rg \geq - \hat{\delta}(r) \, g_{M_j} \},
\end{split}
\end{equation}
where $\la \mathrm{I\!I}_{M_j}(z), \eta \rg$ is the two-form $\la \mathrm{I\!I}_{M_j}(z), \eta \rg(v,w) := \la \mathrm{I\!I}_{M_j}(z) (v,w), \eta \rg$.
We note that $|\eta| \leq 1 + \hat{\omega}(r)$ for every $(z,\eta) \in A_r$, so $A_r \Subset T^{\perp} M_j|_{Q_r}$ because $Q_r$ is compact and the fiber condition is closed.
We form the map $\Phi_r : A_r \to \bR^{n+m}$ by $\Phi_r(z,\eta) := \nabla^{M_j} \hat{U}(z) + \eta$ and define the compact set
\[
G_r := \{ \xi \in \bar{B}^{n+m} :  \exists \; (z,\eta) \in A_r \; \text{ with } \; |\xi - \Phi_r(z,\eta)| \leq \hat{\omega}(r) \}
\]
as the intersection of $\bar{B}^{n+m}$ with the radius-$\hat{\omega}(r)$ tubular neighborhood of $\Phi_r(A_r)$.
\begin{lemma}\label{lemma:BE-lemma-7}
Let $r \in (0, \rho)$.
Then, $U(x) - h(\xi) - \la x, \xi \rg > 0$ for every $x \in Q_r \cap \textup{spt} \, \|V\|$ and every $\xi \in \bar{B}^{n+m} \setminus G_r$.
Consequently, for every Borel set $E \subset Q_r \cap \textup{spt} \, \|V\|$, we have $\mu_f (E) \leq \nu(G_r)$.
In particular, this applies to $E_r = Q_r \cap \cB_j \cap \textup{spt} \, \|V\|$.
\end{lemma}
\begin{proof}
The proof is identical to~\cite{brendle-eichmair}*{Lemma 7}, using Lemma~\ref{lemma:BE-lemma-4} and the definitions of $\hat\omega(r)$ and $\hat\delta(r)$ above.
The final assertion follows from Lemma~\ref{lemma:E-Sigma-perturb}.
\end{proof}

\begin{proposition}\label{prop:laplacian-inequality}
Consider a good point $x_0 \in \textup{spt} \, \|V\|$, a $C^2$ sheet $M_j$ representing $V$ near $x_0$, and a smooth function $\hat{U}$ such that $|\hat{U}(x) - U(x)| = o ( d_j(x,x_0)^2)$ as $x \to x_0$.
Then, it holds that
\begin{equation}\label{eqn:laplacian-inequality}
\begin{split}
     &n \alpha_{\rho}^{- \frac{1}{n}} \Theta^n ( \|V\|, x_0)^{\frac{1}{n}} f(x_0)^{\frac{n}{n-1}} \\
     &\leq f(x_0) \Delta_{M_j} \hat{U}(x_0) + \la \nabla^{M_j} f(x_0), \nabla^{M_j} \hat{U}(x_0) \rg + \textstyle { \sqrt{ |\nabla^{M_j} f(x_0)|^2 + |H_{M_j}(x_0)|^2 f(x_0)^2 } }
\end{split}
\end{equation}
where $\alpha_{\rho}$ is the constant defined in~\eqref{eqn:alpha-rho-constant}.
\end{proposition}
\begin{proof}
Our argument proceeds in two steps. 
First, we prove a local Jacobian inequality for the density of $\|V\|$ at $x_0$; then, we prove~\eqref{eqn:laplacian-inequality} using the arithmetic-geometric mean inequality.
In what follows, we fix the good point $x_0$ and the $C^2$ sheet $M_j$ representing $V$ near $x_0$.

\smallskip \noindent \textbf{Step 1:} We define the set
\[
S(x_0) := \{ y \in T_{x_0}^{\perp} M_j : |\nabla^{M_j} \hat{U}(x_0)|^2 + |y|^2 \leq 1, \; D^2_{M_j} \hat{U}(x_0) - \la \mathrm{I\!I}_{M_j}(x_0), y \rg \geq 0 \}.
\]
Then, we claim that
\begin{equation}\label{eqn:density-bound-step-1}
\begin{split}
\Theta^n ( \|V\|, x_0) \, f(x_0)^{\frac{n}{n-1}} \leq \int_{S(x_0)} \det ( D^2_{M_j} \hat{U}(x_0) - \la \mathrm{I\!I}_{M_j}(x_0), y \rg) \, \rho ( |\nabla^{M_j} \hat{U}(x_0)|^2 + |y|^2) \, dy .
\end{split}
\end{equation}
Let us recall the construction of a $C^2$ diffeomorphism $\Xi^{M_j}_{x_0}$ from a neighborhood of $0 \in T_{x_0} M_j$ to a neighborhood of $x_0 \in M_j$.
We now choose a local trivialization of the normal bundle: for each $\eta \in T_{x_0} M_j$ with $|\eta| < \rho$, we can find a linear isometry $\tau_v^{M_j} : T_{x_0}^{\perp} V \to T^{\perp}_{\Xi^{M_j}_{x_0}(v)} M_j$, such that the map $(v,\eta) \mapsto \tau_v^{M_j}(\eta)$ is of class $C^1$ and $\tau_0^{M_j} = \textup{id}$.
Here, $v \in W_r \subset T_{x_0} M_j$ for $W_r$ the tangent cube considered in~\eqref{eqn:tangent-cube}.
    For each $r \in (0,\rho)$ as above, we define a map $\Psi$ with the properties
    \begin{align*}
    \Psi &: W_r \times T^{\perp}_{x_0} M_j \to \bR^{n+m}, \qquad \Psi(v,\eta) := \Phi_r ( \Xi^{M_j}_{x_0}(v), \tau_v^{M_j}(\eta)) =  \nabla^{M_j} \hat{U}( \Xi^{M_j}_{x_0}(v) ) + \tau_v^{M_j}(\eta), \\
    & | \det D \Psi(0,\eta) |  = | \det ( D^2_{M_j} \hat{U}(x_0) - \la \mathrm{I\!I}_{M_j}(x_0), \eta \rg ) |,
    \end{align*}
for all $\eta \in T_{x_0}^{\perp} M_j$.
Thus, for all $r \in (0,\rho)$ sufficiently small, all $v \in W_r$, and all $|\eta|\leq 2$, we have
\begin{equation}\label{eqn:jacobian-v-computation}
|\det D \Psi(v,\eta)| = | \det ( D^2_{M_j} \hat{U}( \Xi^{M_j}_{x_0}(v) ) - \la \mathrm{I\!I}_{M_j}( \Xi^{M_j}_{x_0}(v) ), \tau_v^{M_j}(\eta) \rg ) \,| + o(1)
\end{equation}
where the term $o(1)$ tends to zero as $r \downarrow 0$.
Next, we decompose the normal space $T_{x_0}^{\perp} M_j$ into compact cubes of size $r$ and let $\cF_r$ denote the collection of all such cubes.
For $k \in \bN^*$, we define
\[
    S_k := \{ y \in T^{\perp}_{x_0} M_j : |\nabla^{M_j} \hat{U}(x_0)|^2 + |y|^2 \leq 1 + k^{-1}, \; D^2_{M_j} \hat{U}(x_0) - \la \mathrm{I\!I}_{M_j}(x_0), y \rg \geq - k^{-1} g_{M_j} \} 
\]
and let $\cF_{r,k} \subset \cF_r$ denote the set of all cubes in $\cF_r$ having non-empty intersection with $S_{2k}$.
We fix $k$ in what follows.
The sets $S_k$ and $S_{2k}$ satisfy $S(x_0) \subset S_{2k} \Subset S_k$ with positive distance, so taking $r_k < \frac{1}{2(m+n)} \textup{dist}(S^c_k, S_{2k})$ ensures that every cube $Q \in \cF_{r,k}$ meeting $S_{2k}$ must be contained in $S_k$.
Moreover, the properties $\lim_{r \downarrow 0} \frac{\hat{\omega}(r)}{r} = \lim_{r \downarrow 0} \hat{\delta}(r) = 0$ from Lemma~\ref{lemma:BE-5-and-6} show that the relevant normal vectors lie in $S_{2k}$, and are therefore in cubes contained in $S_k$: indeed, for $r>0$ small,
\[
(z,y) \in A_r \quad \text{with} \; (z,y) = ( \Xi^{M_j}_{x_0}(v), \tau_v^{M_j}(\eta)), \qquad \implies \quad \eta \in S_{2k}.
\]
This follows from $|\nabla^{M_j} \hat{U}(z)|^2 + |y|^2 \leq (1 + \hat{\omega}(r))^2$ and $\frac{\hat{\omega}(r)}{r} \to 0$, while $\nabla^{M_j} \hat{U}(z) \to \nabla^{M_j} \hat{U}(x_0)$ and $|\eta| = |y|$.
Therefore, $|\nabla^{M_j} \hat{U}(x_0)|^2 + |\eta|^2 \leq 1 + \frac{1}{2k}$ for $r$ sufficiently small.
We similarly write
\[
D^2_{M_j} \hat{U}(z) - \la \mathrm{I\!I}_{M_j}(z), \tau_v^{M_j}(\eta) \rg \geq - \hat{\delta}(r) g_{M_j}(z) \implies D^2_{M_j} \hat{U}(x_0) - \la \mathrm{I\!I}_{M_j}(x_0), \eta \rg \geq - (2k)^{-1} \, g_{M_j}(x_0)
\]
using $\hat{\delta}(r) \to 0$ and continuity.
This proves the claimed containment of cubes, and hence, $\Phi_r(A_r)$ is contained in $\bigcup_{Q \in  \cF_{r,k} } \Psi( W_r \times Q)$ for small $r \in (0,r_k)$, where $A_r$ is defined in~\eqref{eqn:ar-define}.
Consequently,
\begin{align*}
    G_r &= \{ \xi \in \bar{B}^{n+m} : \exists \; (x,y) \in A_r \; \text{ with } \; |\xi - \Phi_r(x,y)| \leq \hat{\omega}(r) \} \\
    &\subset \bigcup_{ Q \in \cF_{r,k} } \{ \xi \in \bar{B}^{n+m} : \exists \; (v,y) \in W_r \times Q \; \text{ with } \; |\xi - \Psi(v,y)| \leq \hat{\omega}(r) \}
\end{align*}
for every $r \in (0,r_k)$ sufficiently small.
Combining this property with~\eqref{eqn:jacobian-v-computation}, we obtain
\begin{align*}
    & \nu ( \{ \xi \in \bar{B}^{n+m} : \exists \; (v,\eta) \in W_r \times Q \; \text{ with } \; |\xi - \Psi(v,\eta)| \leq \hat{\omega}(r) \} ) \\
    &\leq r^n \int_Q \bigl[ | \det ( D^2_{M_j} \hat{U}(x_0) - \la \mathrm{I\!I}_{M_j}(x_0), y \rg ) | \, \rho ( |\nabla^{M_j} \hat{U}(x_0)|^2 + |y|^2) + k^{-1}  \bigr] \, dy.
\end{align*}
for each cube $Q \in \cF_{r,k}$.
Here, we applied the coarea formula~\cite{maggi}*{Ch. 13} and used $\lim_{r \downarrow 0} \frac{\hat{\omega}(r)}{r} = 0$.
Consequently, summing over all cubes $Q \in \cF_{r,k}$ and using $\cF_{r,k} \subset \cF_r$, we obtain
\begin{align*}
    \nu(G_r) & \leq \sum_{Q \in \cF_{r,k}} \nu \, \bigl( \{ \xi \in \bar{B}^{n+m} : \exists \; (v,\eta) \in W_r \times Q \; \text{ with } \; |\xi - \Psi(v,\eta)| \leq \hat{\omega}(r) \} \bigr) \\
    &\leq r^n \int_{S_k} \bigl[ \, \bigl| \det ( D^2_{M_j} \hat{U}(x_0) - \la \mathrm{I\!I}_{M_j}(x_0), y \rg) \bigr| \, \rho ( |\nabla^{M_j} \hat{U}(x_0)|^2 + |y|^2) + k^{-1} \, \bigr] \, dy 
\end{align*}
for every $r \in (0,r_k)$ sufficiently small.
Moreover, we recall from Lemma that~\ref{lemma:BE-lemma-7} that $\mu_f(Q_r \cap \cB_j) \leq \nu(G_r)$, whereas Lemma~\ref{lemma:v-ae-q-gj} with $\mu_f = f^{\frac{n}{n-1}} \|V\|$ and $\theta(x_0) := \Theta^n ( \|V\|,x_0)$ shows that $\mu_f (Q_r \cap \cB_j) = f(x_0)^{\frac{n}{n-1}} \theta(x_0) r^n + o(r^n)$ at a good point $x_0$.
Consequently,
\begin{align*}
    f(x_0)^{\frac{n}{n-1}}\theta(x_0) r^n &= \mu_f ( Q_r \cap \cB_j) + o (r^n) \leq \nu(G_r) + o(r^n) \\
    &\leq r^n \int_{S_k} \bigl[ | \det ( D^2_{M_j} \hat{U}(x_0) - \la \mathrm{I\!I}_{M_j}(x_0), y \rg ) | \, \rho ( |\nabla^{M_j} \hat{U}(x_0)|^2 + |y|^2) + k^{-1}  \bigr] \, dy + o(r^n),
\end{align*}
Dividing by $r^n$ and sending $r \downarrow 0$, we conclude that
\[
f(x_0)^{\frac{n}{n-1}}\theta(x_0) \leq \int_{S_k} \bigl[ | \det ( D^2_{M_j} \hat{U}(x_0) - \la \mathrm{I\!I}_{M_j}(x_0), y \rg ) | \, \rho ( |\nabla^{M_j} \hat{U}(x_0)|^2 + |y|^2) + k^{-1}  \bigr] \, dy.
\]
This property holds for every $k$, while the sets $S_k$ decrease monotonically to $S(x_0)$ in the Hausdorff sense as $k \to \infty$ and the integrands are continuous in $y$ on a fixed compact ball and decrease monotonically.
Thus, we can send $k \to \infty$ and apply the dominated convergence theorem to obtain
\[
\theta(x_0) f(x_0)^{\frac{n}{n-1}} \leq \int_{S(x_0)} | \det ( D^2_{M_j} \hat{U}(x_0) - \la \mathrm{I\!I}_{M_j}(x_0), y \rg ) | \, \rho ( |\nabla^{M_j} \hat{U}(x_0)|^2 + |y|^2) \, dy.
\]
Because the matrix $D^2_{M_j} \hat{U}(x_0) - \la \mathrm{I\!I}_{M_j}(x_0), y \rg \geq 0$ is positive semi-definite on $S(x_0)$, we can remove the absolute value in this last step.
This proves the bound~\eqref{eqn:density-bound-step-1}.

\smallskip \noindent \textbf{Step 2:}
Next, we suppose for contradiction that the inequality~\eqref{eqn:laplacian-inequality} fails at some good point $x_0$, so there exists a real number $\hat{\alpha} > \alpha_{\rho}$ with the property that
\begin{align*}
     & f(x_0) \Delta_{M_j} \hat{U}(x_0) + \la \nabla^{M_j} f(x_0), \nabla^{M_j} \hat{U}(x_0) \rg + \textstyle { \sqrt{ |\nabla^{M_j} f(x_0)|^2 + |H_{M_j}(x_0)|^2 f(x_0)^2 } } \\
     &\leq n \hat{\alpha}^{- \frac{1}{n}} \Theta^n ( \|V\|, x_0)^{\frac{1}{n}} f(x_0)^{\frac{n}{n-1}}.
\end{align*}
Using the Cauchy-Schwarz inequality, we can bound
\begin{align*}
    &\la \nabla^{M_j} f(x_0), \nabla^{M_j} \hat{U}(x_0) \rg + f(x_0) \la H_{M_j}(x_0), y \rg \\
    &\geq  - \sqrt{ |\nabla^{M_j} f(x_0)|^2 + |H_{M_j}(x_0)|^2 f(x_0)^2 } \, \sqrt{ |\nabla^{M_j} \hat{U}(x_0)|^2 + |y|^2 }.
\end{align*}
Using $|\nabla^{M_j} \hat{U}(x_0)|^2 + |y|^2 \leq 1$ for every $y \in S(x_0)$, we therefore bound the above expression by
\begin{align*}
    &f(x_0) ( \Delta_{M_j} \hat{U}(x_0) - \la H_{M_j}(x_0), y \rg ) \\
    &\leq f(x_0) \Delta_{M_j} \hat{U}(x_0) + \la \nabla^{M_j} f(x_0), \nabla^{M_j} \hat{U}(x_0) \rg + \textstyle { \sqrt{ |\nabla^{M_j} f(x_0)|^2 + |H_{M_j}(x_0)|^2 f(x_0)^2 } } \\
     &\leq n \hat{\alpha}^{- \frac{1}{n}} \Theta^n ( \|V\|, x_0)^{\frac{1}{n}} f(x_0)^{\frac{n}{n-1}}.
\end{align*}
Moreover, for every $y \in S(x_0)$, the bilinear form $D^2_{M_j} \hat{U}(x_0) - \la \mathrm{I\!I}_{M_j}(x_0) , y \rg \geq 0$ is positive semidefinite.
Hence, the arithmetic-geometric mean inequality produces
\[
0 \leq \det ( D^2_{M_j} \hat{U}(x_0) - \la \mathrm{I\!I}_{M_j}(x_0), y \rg ) \leq \left( \frac{\Delta_{M_j} \hat{U}(x_0) - \la H_{M_j}(x_0), y \rg }{n} \right)^n \leq \hat{\alpha}^{-1} \Theta^n( \|V\|, x_0) f(x_0)^{\frac{n}{n-1}}.
\]
for every $y \in S(x_0)$, due to the above bound $\Delta_{M_j} \hat{U}(x_0) - \la H_{M_j}(x_0) , y \rg \leq n \hat{\alpha}^{- \frac{1}{n}} \Theta^n (\|V\|, x_0) f(x_0)^{\frac{1}{n-1}}$. 
Thus, applying the inequality~\eqref{eqn:density-bound-step-1}, we obtain
\begin{align*}
    \Theta^n ( \|V\|, x_0) f(x_0)^{\frac{n}{n-1}} &\leq \int_{S(x_0)} \det ( D^2_{M_j} \hat{U}(x_0) - \la \mathrm{I\!I}_{M_j}(x_0), y \rg ) \, \rho ( |\nabla^{M_j} \hat{U}(x_0)|^2 + |y|^2) \, dy \\
    &\leq \hat{\alpha}^{-1} \, \Theta^n( \|V\|, x_0) f(x_0)^{\frac{n}{n-1}} \, \int_{S(x_0)} \rho ( |\nabla^{M_j} \hat{U}(x_0)|^2 + |y|^2) \, dy \\
    &\leq \hat{\alpha}^{-1} \alpha_{\rho} \, \Theta^n( \|V\|, x_0) f(x_0)^{\frac{n}{n-1}}.
\end{align*}
Since $\hat{\alpha} > \alpha_{\rho}$, this produces a contradiction.
Therefore,~\eqref{eqn:laplacian-inequality} holds, and our assertion follows.
\end{proof}

\section{Proof of the main theorem}

We now combine the above observations to prove Theorems~\ref{thm:varifold-isoperimetric} and~\ref{thm:optimal-MSS}.
\begin{proof}[Proof of Proposition~\ref{prop:weighted-MSS}]
We first consider the case $f>0$.
Let $\Omega$ be a bounded open set containing the support of $\|V\|$, and on which $f$ is defined.
Using Proposition~\ref{prop:laplacian-inequality}, we obtain
\begin{align*}
    &n \alpha_{\rho}^{- \frac{1}{n}} \Theta^n ( \|V\|, x)^{\frac{1}{n}} f^{\frac{n}{n-1}} \leq f \Delta_V U + \la \nabla^V f, \nabla^V U \rg + \textstyle { \sqrt{ |\nabla^V f|^2 + |H_V|^2 f^2 } }
\end{align*}
for $\|V\|$-a.e.~ point, where $\Delta_VU(x) := \textup{tr} \, D^2_V U(x)$ is the trace of the Alexandrov Hessian of $U$ restricted to a $C^2$ sheet representing a good point, as in Definition~\ref{def:alexandrov-point}.
We multiply this inequality by some $\varphi \in \textup{Lip}_c(\Omega)$, with $\varphi \geq 0$, and apply Proposition~\ref{prop:fancy-integration-by-parts} with $\psi = \varphi f$.
This gives
\[
\int_{\Omega} \varphi f \Delta_V U \, d\|V\| + \int_{\Omega} \varphi \la \nabla^V f, \nabla^V U \rg \, d\|V\| + \int_{\Omega} f \la \nabla^V \varphi, \nabla^V U \rg \, d\|V\| \leq \int_{\Omega} \varphi f \, d |\delta_s V|.
\]
Moreover, since $U$ is $1$-Lipschitz, the third term of the left-hand side is bounded from below by $- \int_{\Omega} f |\nabla^V \varphi| \, d\|V\|$.
Combining these steps and writing $\theta(x) := \Theta^n (\|V\|,x)$, we obtain
\begin{align*}
        &n \alpha_{\rho}^{- \frac{1}{n}} \int \theta(x)^{\frac{1}{n}} f^{\frac{n}{n-1}} \varphi \, d\|V\| \leq \int \varphi f \, d |\delta_s V| + \int {\textstyle{\sqrt{ |\nabla^V f|^2 + |H_V|^2 f^2} }} \, \varphi \, d \|V\| + \int f |\nabla^V \varphi| \, d \|V\|
\end{align*}
for every non-negative Lipschitz function $\varphi \in \textup{Lip}_c(\Omega)$.
This proves our claim in the case when $f>0$ and $\int_{\Omega} f^{\frac{n}{n-1}} \, d\|V\| = 1$.
The general case $f>0$ follows by scaling.

Next, if $f \geq 0$ is non-negative, it suffices to consider the case when $f>0$ on a set of positive measure, since otherwise the inequality is trivial.
Then, $M := \int_{\Omega} f^{\frac{n}{n-1}} \, d\|V\| >0$.
We define the sequence of functions $f_{\ve} := f+\ve$ for $\ve >0$, so the above result applies to each $f_{\ve}$.
Then, $\nabla^V f_{\ve} = \nabla^V f$ everywhere and $\int_{\Omega} f_{\ve}^{\frac{n}{n-1}} \, d\|V\| \to M$ by the dominated convergence theorem.
For the same reason, $\int_{\Omega} \theta(x)^{\frac{1}{n}} f_{\ve}^{\frac{n}{n-1}} \varphi \, d\|V\| \to \int_{\Omega} \theta(x)^{\frac{1}{n}} f^{\frac{n}{n-1}} \varphi \, d\|V\|$ and
\[
\int_{\Omega} \varphi f_{\ve} \, d|\delta_s V| \to \int_{\Omega} \varphi f \, d |\delta_ sV|, \qquad \int_{\Omega} f_{\ve} \, |\nabla^V \varphi| \, d\|V\| \to \int_{\Omega} f \, |\nabla^V \varphi|\, d\|V\|
\]
because $f \in \textup{Lip}(\Omega)$ is bounded on $\on{spt} \|V\|$, which is compact.
Finally, the middle term satisfies $\sqrt{|\nabla^V f|^2 + |H_V|^2 f_{\ve}^2} \leq |\nabla^V f| + |H_V| (f+\ve)$ with $L^1_{\textup{loc}}(\|V\|)$ right-hand side, hence integrable on $\textup{spt} \, \varphi$.
Thus, we can apply the dominated convergence theorem and take limits in both sides of the inequality for $f_{\ve}$ as $\ve \downarrow 0$ to obtain the claim for any non-negative Lipschitz $f$.

Finally, for an arbitrary Lipschitz function $f$, we apply the preceding discussion to $|f|$, which is Lipschitz and satisfies $|\nabla^V |f|| = |\nabla^V f|$ for $\|V\|$-a.e.~ $x$.
This proves our assertion.
\end{proof}

\begin{proof}[Proof of Theorem~\ref{thm:optimal-MSS}]
        We will choose the density $\rho$ so that nearly all the mass of the measure $\rho ( |\xi|^2) \, d \xi$ on $\bar{B}^{n+m}$ is concentrated near the boundary.
Concretely, for $m,n \geq 2$, we can find a sequence of continuous functions $\rho_k: [0,\infty) \to (0,\infty)$ such that for $k \to \infty$, we have
\[
        \int_{\bar{B}^{n+m}} \rho_k ( |\xi|^2) \, d \xi = 1, \qquad \sup_{ [ 0, 1 - k^{-1}] } \rho_k \leq o(1), \qquad \sup_{ [ 1 - k^{-1} , 1] } \rho_k \leq \frac{2k}{(n+m ) |B^{n+m}|} + o(k).
\]
For each point $z \in \bR^n$, we obtain
\begin{align*}
    & \frac{1}{|B^m|} \int_{ \{ y \in \bR^m : |z|^2 + |y|^2 \leq 1 \} } \rho_k ( |z|^2 + |y|^2) \, dy \\
    &\leq (1 - |z|^2 - k^{-1})^{\frac{m}{2}}_+ \sup_{ [0,1 - k^{-1}] } \rho_k + \bigl[ (1 - |z|^2)^{\frac{m}{2}}_+ - (1 - |z|^2 - k^{-1})_+^{\frac{m}{2}} \bigr] \sup_{ [1 - k^{-1}, 1]} \rho_k \\
    &\leq \sup_{[0,1 - k^{-1}]} \rho_k + \tfrac{m}{2} k^{-1} \sup_{[1 - k^{-1},1]} \rho_k.
\end{align*}
Sending $k \to \infty$, we see that the quantity $\alpha_k := \alpha_{\rho_k}$ from~\eqref{eqn:alpha-rho-constant} satisfies $\alpha_k \leq \frac{m |B^m|}{(m+n) |B^{m+n}|} + o(1)$; here, we used $m \geq 2$.
Applying Proposition~\ref{prop:weighted-MSS}, we conclude that
\begin{align*}
    &n \Bigl( \frac{(n+m) |B^{n+m}|}{m |B^m|} \Bigr)^{\frac{1}{n}} \Bigl( \int_{\Omega} f^{\frac{n}{n-1}} \, d\|V\| \Bigr)^{- \frac{1}{n}} \int_{\Omega} \Theta^n ( \|V\|, x)^{\frac{1}{n}} f(x)^{\frac{n}{n-1}} \varphi(x) \, d \|V\|(x) \\
    &\leq \int_{\Omega} \varphi \sqrt{ |\nabla^V f|^2 + |H_V|^2 f^2 } \, d \|V\| + \int_{\Omega} f \varphi \, d |\delta_s V| + \int_{\Omega} f |\nabla^V \varphi| \, d \|V\|.
\end{align*}
Since $\textup{spt} \, \|V\|$ is compact, we can take $\varphi \in C_c^{\infty}(\Omega)$ to satisfy $\varphi \equiv 1$ on a neighborhood of this set, so the last term vanishes and the desired inequality follows.
This completes the proof.
\end{proof}

\begin{proof}[Proof of Theorem~\ref{thm:varifold-isoperimetric}]
    This result follows from Theorem~\ref{thm:optimal-MSS} upon taking $f=1$ and using the fact that integral varifolds have $\Theta^n( \|V\|,x) \in \bN^*$ for $\|V\|$-a.e.~ $x$, so $\Theta^n (\|V\|,x) \geq 1$ for $\|V\|$-a.e.~ $x$.
\end{proof}

\section{Rigidity of the isoperimetric inequality}\label{section:rigidity}

We now prove that equality in the sharp codimension-$2$ isoperimetric inequality for varifolds is attained precisely by flat round $n$-disks.
Let $V$ be a compactly supported $n$-varifold on $\bR^{n+2}$ attaining equality in the bound of Theorem~\ref{thm:varifold-isoperimetric}, scaled to satisfy $\|V\|(\bR^{n+2}) = 1$.

\begin{lemma}\label{lemma:indecomposable}
The varifold $V$ has density $\Theta^n (\|V\|,x) = 1$ for $\|V\|$-a.e.~ $x$.
Moreover, $V$ is indecomposable in the sense of~\cite{weakly-differentiable-menne}*{Definition 6.2}.
\end{lemma}
\begin{proof}
Using Theorem~\ref{thm:optimal-MSS} with $f=1$ and $\Theta^n ( \|V\|,x) \geq 1$ for $\|V\|$-a.e.~ $x$ by integrality, we find
    \[
    \|V\| ( \bR^{n+2}) = \int \Theta^n ( \|V\|,x)^{\frac{1}{n}} \, d\|V\|(x) \leq n^{-1} |B^n|^{- \frac{1}{n}} \|V\|( \bR^{n+2})^{\frac{1}{n}} \| \delta V\| ( \bR^{n+2}) = \|V\| ( \bR^{n+2}).
    \]
    Therefore, equality forces $\Theta^n ( \|V\|,x) = 1$ for $\|V\|$-a.e.~ $x$.
    This proves our first assertion.

    Because $V$ is integral, compactly supported, and has locally bounded first variation, it admits a decomposition $V = \sum_{\beta} V_{\beta}$ by~\cite{weakly-differentiable-menne}*{Theorem 6.12}, which satisfies $\|V\| = \sum_{\beta} \|V_{\beta}\|$ and $\| \delta V\| = \sum_{\beta} \| \delta V_{\beta}\|$.
    The sharp inequality of Theorem~\ref{thm:varifold-isoperimetric} for each component $V_{\beta}$ implies
    \[
    n |B^n|^{\frac{1}{n}} \,\|V\|(\bR^{n+2})^{\frac{n-1}{n}} = \| \delta V\| ( \bR^{n+2}) = \textstyle{ \sum_{\beta} \| \delta V_{\beta} \| ( \bR^{n+2}) } \geq n |B^n|^{\frac{1}{n}} \, { \textstyle \sum_{\beta} \| V_{\beta}\| (\bR^{n+2})^{\frac{n-1}{n}} }
    \]
    If there were at least two non-trivial components, then the strict concavity of the function $t \mapsto t^{\frac{n-1}{n}}$, with $\sum_{\beta} \|V_{\beta}\| = \|V\|$, would produce a contradiction.
    Thus, $V$ is indecomposable as claimed.
\end{proof}

Following the arguments of Section~\ref{section:optimal-transportation-varifolds}, with $f=1$, we consider a sequence of optimal pairs $\{ (U_k, h_k) \}$ satisfying~\eqref{eqn:u-h-property} and~\eqref{eqn:u-property} for a sequence of densities $\rho_k$ with $\sup_{ [0, 1 - k^{-1}] } \rho_k \leq o(1)$.
The optimal data remain admissible under the transformation $(\tilde{U}_k, \tilde{h}_k) := (U_k - c_k, h_k - c_k)$, with $c_k = U_k(x_0)$ for some fixed $x_0 \in \bR^{n+2}$.
The functions $\tilde{U}_k$ are $1$-Lipschitz and satisfy $\tilde{U}_k(x_0)=0$, so the Arzel\`a-Ascoli theorem and a diagonal argument produce a $1$-Lipschitz, convex subsequential limit function $U$ on $\bR^{n+2}$, with $\tilde{U}_{k'} \to U$ locally uniformly.
On each $C^2$ sheet $M_j$, the restrictions $\tilde{U}_{k'}|_{M_j}$ are locally uniformly semiconvex; in a $C^2$ coordinate chart, after adding a fixed quadratic function, we can apply the convergence theorem for gradients of convex functions~\cite{rockafellar}*{Theorem 24.5}.
Thus, $\nabla^V U_{k'} \to \nabla^V U$ for $\|V\|$-a.e. point, and the uniform Lipschitz bound together with dominated convergence gives $\nabla^V U_{k'} \to \nabla^V U$ strongly in $L^p_{\textup{loc}}(\|V\|)$ for every $1 \leq p < \infty$.
We relabel the normalized subsequence as $U_k$.
Moreover, the measures $\nu_k$ weak-$\ast$ converge to the normalized surface measure $\sigma_{\bS^{n+1}}$ on $\partial \bar{B}^{n+2}$.
%By the compactness of the space of probability measures, after passing to a further subsequence, we can extract a limit measure $\gamma_k \xrightharpoonup{*} \gamma \in \cP ( \textup{spt} \, \|V\| \times \bar{B}^{n+2})$ with $\gamma \in \Pi (\|V\|, \sigma_{\bS^{n+1}})$ due to $(\pi_{\textup{spt} \, \|V\|})_{\#} \gamma = \|V\|$ and $( \pi_{\bar{B}^{n+2}})_{\#} \gamma = \sigma_{\bS^{n+1}}$.
%Moreover, $\gamma$ is the optimal plan with data $(U,h)$, for $\tilde{h}_{k''} \to h$ a further subsequential limit as above, and $\gamma$ is concentrated on the contact set $\Gamma_U = \{ (x,\xi) \in \textup{spt} \, \|V\| \times \bar{B}^{n+2} : U(x) - h(\xi) - \la x, \xi \rg = 0 \}$.
The set of good points for each $U_k$ has full $\|V\|$-measure, hence so does the set of points that are jointly good for all the functions $U_k$.

\begin{lemma}\label{lemma:positive-measure-subset}
For any $E \subset \textup{spt} \, \|V\|$ with $\|V\|(E)>0$, there exists a compact set $K \subset E$ contained in a single $C^2$ Borel set $\cB_j \subset M_j$ of $V$ such that $\|V\|(K)>0$, the map $x \mapsto \nabla^V U(x)$ is Lipschitz on $K$, and after passing to a further subsequence, $\nabla^V U_{k'} \to \nabla^V U$ uniformly in $K$.
\end{lemma}
\begin{proof}
Let $\cG$ denote the full $\|V\|$-measure set of point which are good for $U$ and for every $U_k$, and at which the convergence $\nabla^V U_k \to \nabla^V U$ holds after passing to the subsequence fixed above.
Since the countably many $C^2$ sheets cover $V$ up to a $\|V\|$-null set, there is a sheet $M_j$, a relatively compact coordinate patch $\Omega_j \Subset M_j$, and a Borel set $\cB_j \subset M_j$ such that $\|V\|(E \cap \cB_j \cap \Omega_j \cap \cG) > 0$.

On $M_j$, the restriction $U|_{M_j}$ is semiconvex, so $\nabla^{M_j} (U|_{M_j})$ is approximately differentiable at $\cH^n$-a.e.~ Alexandrov point of $M_j$ in the sense of Federer~\cite{federer-gmt}*{\S 3.1.2}.
Equivalently, by the Lusin $C^1$ approximation theorem for approximately differentiable maps, the set of approximate differentiability points can be covered, up to a $\cH^n$-null et, by countably many sets on which $\nabla^{M_j}(U|_{M_j})$ agrees with a $C^1$ map; see~\cite{weakly-differentiable-menne}*{Lemma 11.1 and Theorem 11.2} and the $C^1$ approximation property~\cite{simon-gmt}*{Theorem 1.5}.
One such set meets $E\cap\cB_j\cap\Omega_j\cap\cG$ in positive $\|V\|$-measure; by inner regularity we choose a compact subset $K_0$ of positive measure on which $\nabla^VU$ is Lipschitz.
Using Egorov's theorem, and once again inner regularity, we obtain a compact set $K\subset K_0$ with $\|V\|(K)>0$ on which $\nabla^VU_k\to\nabla^VU$ uniformly.
This proves the assertion.
\end{proof}

\begin{lemma}\label{lemma:gradient-bound}
The limit function $U$ satisfies $|\nabla^V U(x)| < 1$ for $\|V\|$-a.e.~ $x$.
\end{lemma}
\begin{proof}
    We argue by contradiction: suppose that the set $E := \{ x \in \on{spt} \|V\| : |\nabla^V U(x)| = 1 \}$ has positive $\|V\|$-measure.
    By Lemma~\ref{lemma:positive-measure-subset},  there is a compact set $K \subset E$ contained in a $C^2$ coordinate patch inside a Borel set $\cB_j \subset M_j$, containing points that are jointly good for the functions $\{ U_k \}$ and $U$, such that $\|V\|(K)>0$, $\nabla^V U_k \to \nabla^V U$ uniformly, and $x \mapsto \nabla^V U(x)$ is Lipschitz on $K$.
    We set
    \[
    \ve_k := \sup_{x \in K} |\nabla^V U_k(x) - \nabla^V U(x)| + 2 \, \sup_{x \in K} |\nabla^V U_k(x) - \nabla^V U(x)|^{\frac{1}{2}}, \qquad \ve_k \to 0 \quad \text{as } \; k \to \infty.
    \]
    We now claim that any contact set point $(x,\xi) \in \Gamma_{U_k}$ with $x \in K$ has $\textup{dist}( \xi, \nabla^V U(K)) \leq \ve_k$.
    By Lemma~\ref{lemma:BE-lemma-4}, the tangential projection $\xi^{\top} = \pi_{T_x M_j} \xi \in \partial_{M_j} U_k(x)$.
    Since $U_k|_{M_j}$ is differentiable at $x$, the tangential projection of $\partial_{M_j} U_k$ consists of the single vector $\nabla^V U_k(x)$, hence $\xi^{\top} = \nabla^V U_k(x)$.
    Using $|\xi| \leq 1$ and $|\nabla^V U(x)|  = 1$ on $K$, we obtain $|\xi^{\perp}|^2 \leq 1 - |\xi^{\top}|^2 \leq 1 - |\nabla^V U_k(x)|^2$.
    On the other hand, $|\nabla^V U_k(x)| \geq 1 - |\nabla^V U_k(x) - \nabla^V U(x)|$.
    Consequently, 
    \begin{align*}
    |\xi^{\perp}|^2 &\leq 1 - (1 - |\nabla^V U_k(x) - \nabla^V U(x)|)^2 \leq 2 \, |\nabla^V U_k(x) - \nabla^V U(x)|, \\
    |\xi - \nabla^V U(x)| &\leq |\xi^{\top} - \nabla^V U(x)| + |\xi^{\perp}| = |\nabla^V U_k - \nabla^V U| + |\xi^{\perp}| \leq \ve_k.
    \end{align*}
    Thus, for $G_k := \{ \xi \in \bar{B}^{n+m} : \textup{dist}( \xi, \nabla^V U(K)) \leq \ve_k \}$, we have $U_k(x) - h_k(\xi) - \la x, \xi \rg > 0$ for every $(x,\xi) \in K \times (\bar{B}^{n+m} \setminus G_k)$, and Lemma~\ref{lemma:E-Sigma-perturb} implies $\|V\|(K) \leq \nu_k(G_k)$ for $d \nu_k(\xi) = \rho_k(|\xi|^2) \, d \xi$.
    
    On the other hand, since the function $\nabla^V U$ is Lipschitz on $K$, the set $\nabla^V U(K)$ is $n$-rectifiable and satisfies $\nabla^V U(K) \subset \bS^{n+m-1}$, because $K \subset E$ and $|\nabla^V U| = 1$ on $E$.
    Since $m \geq 2$, this implies $\cH^{n+m-1} ( \nabla^V U(K)) = 0$.
    Because the measures $\nu_k$ are radial probability measures, working in polar coordinates shows that for every Borel set $A \subset \bS^{m+n-1}$, we have
    \[
    \nu_k ( \{ r \omega : 0 \leq r \leq 1, \; \omega \in A \} ) = C_{n,m}\cH^{n+m-1}(A), \qquad C_{n,m} := \cH^{n+m-1}(\bS^{n+m-1})^{-1}.
    \]
    Therefore, $\xi = r \omega \in \bar{B}^{n+m}$ has $\textup{dist} ( \xi, \nabla^VU(K)) \leq \ve_k$, so for all $k$ sufficiently large, we obtain $\textup{dist}_{\bS^{n+m-1}} ( \omega, \nabla^V U(K)) \leq C \ve_k$.
    This implies that
    \[
    \nu_k(G_k) \leq C_{n,m} \, \cH^{n+m-1} ( \{ \omega \in \bS^{n+m-1} : \textup{dist}_{\bS^{n+m-1}} ( \omega, \nabla^V U(K)) \leq C \ve_k\}  ) .
    \]
    However, the set $\nabla^V U(K)$ is compact and has zero $\cH^{n+m-1}$-measure, so the measures of the above $\ve_k$-neighborhoods tend to zero as $\ve_k \to 0$ and $0 < \|V\|(K) \leq \nu_k(G_k) \to 0$, a contradiction.
    Thus, $\|V\|(E) = 0$ and $|\nabla^V U(x)| < 1$ for $\|V\|$-a.e.~ $x$.
    This completes the proof of our assertion.
\end{proof}

\begin{lemma}\label{lemma:converge-to-uniform-measure}
Consider a sequence of densities $\rho_k : [0,\infty) \to [0,\infty)$ with $\sup_{ [0, 1 - k^{-1}] } \rho_k \leq o(1)$ and a sequence of vectors $p_k \in \bR^n$ with $p_k \to p$, where $|p|<1$.
We define a finite measure $\tilde{\nu}_k$ on $\bR^2$ by
\[
d \tilde{\nu}_k(y) := \alpha_k^{-1} \, \rho_k ( |p_k|^2 + |y|^2) \, \mathbf{1}_{ \{ |p_k|^2 + |y|^2 \leq 1 \} } \, dy.
\]
Then, we have $\tilde{\nu}_k \xrightharpoonup{*} \tilde{\nu}_{\infty}$ weakly as Radon measures on $\bR^2$, where $\tilde{\nu}_{\infty}$ denotes the uniform probability measure on the circle $\{ y \in \bR^2 : |y| = \sqrt{1 - |p|^2} \}$.
\end{lemma}
\begin{proof}
We use polar coordinates $y = s \omega$ with $\omega \in \bS^1$ and write $d \bar{\omega}$ for the normalized probability measure on $\bS^1$, so $d y = 2 \pi s \, ds \, d \bar{\omega}$.
Let $t = |p_k|^2 + s^2$, so $s \, ds = \frac{1}{2} dt$ and any $\varphi \in C_c(\bR^2)$ satisfies
\begin{align*}
\int_{\bR^2} \varphi \, d \tilde{\nu}_k &= \frac{1}{\alpha_k} \int_{ \{ |p_k|^2 + |y|^2 \leq 1 \} } \varphi(y) \rho_k ( |p_k|^2 + |y|^2) \, dy = \frac{\pi}{\alpha_k} \int_{ |p_k|^2 }^1 \rho_k(t) \int_{\bS^1} \varphi \bigl( \sqrt{ t - |p_k|^2 } \, \omega \bigr) \, d \bar{\omega} \, dt \\
&= \bigl( \int_0^1 \rho_k(t) \, dt \Bigr)^{-1} \int_{|p_k|^2}^1 \rho_k(t) \int_{\bS^1} \varphi( { \textstyle \sqrt{t - |p_k|^2} } \, \omega) \, d \bar{\omega} \, dt .
\end{align*}
In the last step, we used the fact that $\alpha_k = \pi \int_0^1 \rho_k(t) \, dt$ for $m=2$.
We now observe that
\[
\int_0^1 \rho_k(t) \, dt \geq \int_0^1 \rho_k(t) t^{\frac{n}{2}} \,dt = \frac{2}{(n+2) |B^{n+2}|} \int_{\bar{B}^{n+2}} \rho_k ( |\xi|^2) \, d \xi = \frac{2}{(n+2) |B^{n+2}|}.
\]
Thus, the normalized measures $\frac{\rho_k(t) \, dt}{\int_0^1 \rho_k(s) \, ds}$ have denominator bounded below while $\sup_{ [0, 1- k^{-1}] } \rho_k = o(1)$, so they converge weakly to $\delta_1$ on $[0,1]$ because $|p_k| \to |p| < 1$.
Let $\Phi_k(t) = \int_{\bS^1} \varphi( \sqrt{t - |p_k|^2} \, \omega) \, d \bar{\omega}$, for $t \in [ |p_k|^2, 1]$, and $\Phi_{\infty} = \int_{\bS^1} \varphi ( \sqrt{1 - |p|^2} \, \omega) \, d \bar{\omega}$, so the uniform continuity of $\varphi$ implies that
\begin{equation}\label{eqn:small-eps-Phi-infty}
\lim_{\ve \downarrow 0} \limsup_{k \to \infty} \sup_{ t \in [1-\ve, 1] } \Bigl| \int_{\bS^1} \varphi ( \textstyle{ \sqrt{t - |p_k|^2} } \, \omega) \, d \bar{\omega} - \int_{\bS^1} \varphi ( \textstyle{ \sqrt{1-|p_k|^2} } \, \omega) \, d \bar{\omega} \Bigr| = 0.
\end{equation}
We may therefore express
\[
\int_{|p_k|^2}^1 \rho_k(t) \Phi_k(t) \, dt = \int_{|p_k|^2}^{1-\ve} \rho_k(t) \Phi_k(t) \, dt + \int_{1-\ve}^1 \rho_k(t) \Phi_k(t) \, dt.
\]
We send $k \to \infty$, so the first region contributes $o(1)$ due to $\frac{\rho_k(t) \, dt}{\int_0^1 \rho_k(s) \, ds} \to \delta_1$.
In view of~\eqref{eqn:small-eps-Phi-infty}, the second integral is uniformly close to $\Phi_{\infty}$ after sending $k \to \infty$ followed by $\ve \downarrow 0$.
Combining these computations, we obtain $\lim_{k \to \infty} \int_{\bR^2} \varphi(y) \, d \tilde{\nu}_k(y) = \Phi_{\infty} = \int_{\bS^1} \varphi ( \sqrt{1 - |p|^2} \, \omega) d \bar{\omega}$.
This implies the claimed convergence $\tilde{\nu}_k \xrightharpoonup{*} \tilde{\nu}_{\infty}$, completing the proof of our assertion.
\end{proof}

We conclude from the above discussion that there is a set of full $\|V\|$-measure consisting of the $x \in \bR^{n+2}$ that are jointly good points for $U$ and each of the functions $U_k$, with $\Theta^n( \|V\|,x) = 1$ and
\[
U_k(x) \to U(x), \qquad \nabla^V U_k(x) \to \nabla^V U(x), \qquad \text{and} \qquad |\nabla^V U(x)|<1.
\]
\begin{proposition}\label{prop:reduce-to-nice-configuration}
For every compactly supported Lipschitz function $\psi$, we have
\begin{equation}\label{eqn:equality-limit-psi-theta(x)}
    n |B^n|^{\frac{1}{n}} \int \psi \, d \|V \| = - \int \la \nabla^V \psi, \nabla^V U \rg \, d \|V\| + \int \psi \, d |\delta_s V|.
\end{equation}
Moreover, for $\|V\|$-a.e.~ $x$ we have that $H_V=0$, $\mathrm{I\!I}_V =0$, $D^2_V U = |B^n|^{\frac{1}{n}} \textup{id}_{T_x V}$, and
\begin{equation}\label{eqn:subdifferential-U}
    \partial U(x) = \nabla^V U(x) + \{ y \in T_x^{\perp} V : |y| \leq \textstyle{ \sqrt{1 - |\nabla^V U(x)|^2} } \}. 
\end{equation}
\end{proposition}
\begin{proof}
We proceed in a number of steps.

\smallskip \noindent \textbf{Step 1: The Laplacian.}
    Using Proposition~\ref{prop:laplacian-inequality} and the property $\Theta^n ( \|V\|,x) = 1$ a.e., we have
    \[
    F_k(x) := \Delta_V U_k(x) + |H_V(x)| - n \alpha_k^{- \frac{1}{n}} \geq 0
    \]
    at every good point $x \in \on{spt} \|V\|$.
    Recall that $F_k \in L^1( \|V\|)$ because the Alexandrov Hessian of $U_k|_{M_j}$ is a measurable symmetric tensor field on each Borel sheet $\cB_j \subset M_j$, while $\Delta_V U_k \geq - |H_V|$ with $H_V \in L^1(\|V\|)$.
    Also, the integration by parts inequality of Proposition~\ref{prop:fancy-integration-by-parts} gives an upper bound on $\int \Delta_V U_k \, d\|V\|$, so the negative part of $\Delta_V U_k$ is integrable and the integral is bounded from above; thus, $\Delta_V U_k \in L^1(\|V\|)$.
    Moreover, with $\|V\|(\bR^{n+2}) = 1$ and $\alpha_k \to |B^n|^{-1}$, we find
    \begin{align*}
        \Bigl[ |\delta_s V| (\bR^{n+2}) - \int \Delta_V U_k \, d \|V\| \Bigr] + \int F_k \, d\|V\| = \| \delta V\|(\bR^{n+m}) - n \alpha_k^{- \frac{1}{n}} \to 0 \quad \text{as } \; k \to \infty.
    \end{align*}
    Taking a sequence of compactly supported functions $\varphi$ converging to $1$ in Proposition~\ref{prop:fancy-integration-by-parts}, we obtain $\int_{\Omega} \Delta_V U_k \, d \|V\| \leq |\delta_s V|(\Omega)$ for each $U_k$, so both left-hand side summands are non-negative, hence $\int F_k \, d \|V\| \to 0$ and $\int \Delta_V U_k \, d\|V\| \to |\delta_s V|(\bR^{n+2}) $. 
    Because $F_k \in L^1 (\|V\|)$ are non-negative functions, after passing to a further subsequence, this implies that $F_{k} \to 0$ for $\|V\|$-a.e.~ $x$, and hence 
    \begin{equation}\label{eqn:convergence-laplacian}
        \Delta_V U_k + |H_V| \to n |B^n|^{\frac{1}{n}} \qquad \text{for } \; \|V\| \text{-a.e.~ } x.
    \end{equation}

\noindent \textbf{Step 2: Mean curvature and second fundamental form.}
    At a point $x$ that is jointly good for $U$ and each $U_k$ and satisfies all the $\|V\|$-a.e.~ properties deduced above, we let $M_j$ be a representing $C^2$ sheet for $V$ at such a point and we denote by $S(U_k;x)$ the set $S(x)$ from Step 1 of Proposition~\ref{prop:laplacian-inequality} adapted to the function $U_k$. 
    The inequality~\eqref{eqn:density-bound-step-1} therein shows, in particular, that $S(U_k;x_0) \neq \varnothing$ for $\|V\|$-a.e.~ jointly good point $x_0$ and every $k$, so we can find a $y_k$ with $|y_k| \leq 1$ and $D^2_{M_j} U_k(x) \geq \la \mathrm{I\!I}_{M_j}(x), y_k \rg \geq - |\mathrm{I\!I}_{M_j}(x)|$.
    Combined with the property~\eqref{eqn:convergence-laplacian}, this implies uniform upper and lower bounds on $\textup{tr}(D^2_{M_j} U_k(x))$ depending only on $\mathrm{I\!I}_{M_j}(x)$.
    Thus, the matrices $D^2_{M_j} U_k(x) \in \textup{Sym}^2 (T_x^{\vee} V)$ are uniformly bounded for $\|V\|$-a.e.~ $x$, so we can extract a subsequential limit $A \in \textup{Sym}^2(T_x^{\vee} V)$.
    At a good point $x$, the role of the smooth function $\hat{U}$ in Proposition~\ref{prop:laplacian-inequality} is played by any smooth second-order representative of the Alexandrov expansion of $U_k|_{M_j}$ at $x$, so we can write $\nabla^{M_j} U_k(x)$ and $D^2_{M_j}(x)$ in place of the corresponding derivatives of $\hat{U}$.
    We also write
    \begin{align*}
    \det\!{}_+ L &= \det L \quad \text{for } \; L \geq 0, \qquad \det\!{}_+ L = 0 \quad \text{otherwise}, \\
    f_k(y) &= n^n \bigl(\textup{tr} \, D^2_{M_j} U_k(x) + |H_V(x)| \bigr)^{-n} \det\!{}_+ \bigl( D^2_{M_j} U_k(x) - \la \mathrm{I\!I}_{M_j}(x), y \rg \bigr), \qquad \text{for each } \; k,
    \end{align*}
    which is a continuous function of $y$, by the continuity of $\det_+$.
    In terms of the measures $\tilde{\nu}_k$ on $T_x^{\perp} V$, the inequality~\eqref{eqn:density-bound-step-1}, applied to each sufficiently large $k$ and using $\Theta^n ( \|V\|,x) = 1$, assumes the form
        \[
    n^n \alpha_k^{-1} \, \bigl( \textup{tr} \, D^2_{M_j} U_k(x) + |H_V(x)| \bigr)^{-n} \leq \int_{ \{ |p_k|^2 + |y|^2 \leq 1 \} } f_k \, d \tilde{\nu}_k \leq 1.
    \]
    Here, we used~\eqref{eqn:convergence-laplacian} to find $\Delta_{M_j} U_k(x) + |H_V(x)| > 0$ for all sufficiently large $k$, so if $D^2_{M_j} U_k(x) - \la \mathrm{I\!I}_{M_j}(x), y \rg \geq 0$, the arithmetic-geometric mean inequality gives $0 \leq f_k (y) \leq 1$ due to 
    \[
    \det ( D^2_{M_j} U_k(x) - \la \mathrm{I\!I}_{M_j}(x) , y \rg ) \leq n^{-n} ( \textup{tr} \, D^2_{M_j} U_k(x) - \la H_V(x), y \rg )^n
    \]
    and $- \la H_V , y \rg \leq |H_V| \cdot |y| \leq |H_V|$ since $|y| \leq 1$ on $\textup{spt} \, \tilde{\nu}_k$. 
    Because $D^2_{M_j} U_k(x) \to A$, we obtain $f_k(y) \to f(y)$ uniformly, where $f(y) = \frac{1}{|B^n|} \det\!{}_+ (A - \la \mathrm{I\!I}_{M_j}(x), y \rg )$.
    The same inequality shows that on the circle $\{ |y| = \sqrt{1 - |\nabla^V U|^2} \}$, we have $f(y) = 0$ if $A - \la \mathrm{I\!I}_{M_j}(x), y \rg$ is not positive semidefinite, and $0 \leq f(y) \leq 1$ otherwise, due to
    \begin{equation}\label{eqn:kappa-n-inequality}
        \det ( A - \la \mathrm{I\!I}_{M_j}(x), y \rg ) \leq \Bigl( \frac{\on{tr} A - \la H_V(x), y \rg}{n} \Bigr)^n \leq \Bigl( \frac{\on{tr} A + |H_V(x)|}{n} \Bigr)^n = |B^n|.
    \end{equation}
    We now apply Lemma~\ref{lemma:converge-to-uniform-measure} with $p_k = \nabla^{M_j} U_k \to p = \nabla^{M_j} U$ for a jointly good point $x$, so $|p|<1$ and the measures $\tilde{\nu}_k$ on the two-dimensional vector space $T_x^{\perp} V$ satisfy $\tilde{\nu}_k \xrightharpoonup{*} \tilde{\nu}_{\infty}$, with $\tilde{\nu}_{\infty}$ the uniform measure on the circle $\{ y \in T_x^{\perp} V : |y| = \sqrt{1 - |p|^2} \}$.
    Sending $k \to \infty$, the property~\eqref{eqn:convergence-laplacian} shows that the left-hand side tends to $1$, and hence $\int f_k \, d \tilde{\nu}_k \to 1$.
    Using $f_k \to f$ and $\tilde{\nu}_k \xrightarrow{*} \tilde{\nu}_{\infty}$, this leads to $\int f \, d\tilde{\nu}_{\infty} = 1$.
    Because $0 \leq f(y) \leq 1$ on the limiting circle and $\tilde{\nu}_{\infty}$ is a probability measure, this forces $f(y) = 1$, $\tilde{\nu}_{\infty}$-a.e.~ $y$ on $\{ |y| = \sqrt{1 - |\nabla^V U|^2} \}$.
    By the continuity of $f$, this implies $f(y)=1$ for all such $y$, hence $\det_+ ( A - \la \mathrm{I\!I}_{M_j}(x), y \rg) = |B^n| > 0$.
    Thus, the matrix $A - \la \mathrm{I\!I}_{M_j}(x), y \rg$ is positive-definite and attains equality in all steps of the inequality~\eqref{eqn:kappa-n-inequality}, whereby
    \begin{equation}\label{eqn:A-identity-y}
    A - \la \mathrm{I\!I}_{M_j}(x), y \rg = |B^n|^{\frac{1}{n}}  \, \textup{id}_{T_x V}
    \end{equation}
    for every $y \in T_x^{\perp} V$ with $|y| = \sqrt{1 - |\nabla^V U|^2}$.
    Using $\textup{tr}(A) = n |B^n|^{\frac{1}{n}} - |H_V(x)|$ from~\eqref{eqn:convergence-laplacian}, we obtain $\la H_V(x), y \rg = - |H_V(x)|$ for every $y$ on this circle; thus, $H_V(x)=0$.
    Also, taking $y \mapsto -y$ in the equality~\eqref{eqn:A-identity-y}, we obtain $\la \mathrm{I\!I}_{M_j}(x), y \rg = 0$ for every vector on a circle of positive radius in $T_x^{\perp} V$, so $\mathrm{I\!I}_{M_j}(x) = 0$ and $A = |B^n|^{\frac{1}{n}} \, \textup{id}_{T_x V}$.
    Thus, every subsequential limit of the bounded sequence $D^2_V U_k(x)$ equals $A = |B^n|^{\frac{1}{n}} \, \textup{id}_{T_x V}$, so
    the whole sequence converges and $D^2_V U_k(x) \to |B^n|^{\frac{1}{n}} \textup{id}_{T_x V}$ at $\|V\|$-a.e.~ jointly good point $x$.
    Moreover, $H_V =0$ and $\mathrm{I\!I}_V =0$ proves the first part of the claim.

    \smallskip \noindent \textbf{Step 3: The integration by parts.}
    By the above convergence of $D^2_VU_k$ for $\|V\|$-a.e.~ $x$ and $H_V=0$, we have $\delta V = \delta_s V$ and $|\delta_s V|(\bR^{n+2}) = n |B^n|^{\frac{1}{n}}$.
    The convergence~\eqref{eqn:convergence-laplacian} becomes $\int \Delta_V U_k \, d\|V\| \to |\delta_s V| ( \bR^{n+2})$ and $\Delta_V U_k \to n |B^n|^{\frac{1}{n}}$ for $\|V\|$-a.e.~ $x$.
    By Step 1, the functions $\Delta_V U_k \in L^1( \|V\|)$ are uniformly bounded below by an $L^1$ function, converge a.e., and their integrals converge to the integral of the limit, hence Vitali's theorem gives $\Delta_V U_k \to n |B^n|^{\frac{1}{n}}$ strongly in $L^1( \|V\|)$.
    Next, we observe that if $A_k, A$ are positive semidefinite symmetric $n \times n$ matrices, then
    \begin{equation}\label{eqn:Ak-A-inequality}
        |A_k - A| \leq C_n ( \on{tr} A_k + \on{tr} A).
    \end{equation}
    Indeed, a symmetric matrix with $A \geq 0$ and eigenvalues $\lambda_1, \dots, \lambda_n \geq 0$ has Hilbert-Schmidt norm $|A|_{\textup{HS}}^2 = \sum_{i=1}^n \lambda_i^2 \leq ( \sum_{i=1}^n \lambda_i)^2 = ( \on{tr} A)^2$, hence $|A|_{\textup{HS}} \leq \on{tr} A$.
    Thus, by the triangle inequality, 
    \[
    |A_k - A|_{\textup{HS}} \leq |A_k|_{\textup{HS}} + |A|_{\textup{HS}} \leq \on{tr} A_k + \on{tr} A.
    \]
    The bound~\eqref{eqn:Ak-A-inequality} now follows from the equivalence of norms on a finite-dimensional vector space.

    In our situation, the matrices $A_k(x) = D^2_V U_k(x)$ are positive semidefinite and satify $A_k(x) \to A(x) =: |B^n|^{\frac{1}{n}} \textup{id}_{T_x V}$ for $\|V\|$-a.e.~ $x$ and $\textup{tr} \, A_k(x) \to \textup{tr} \, A(x) = n |B^n|^{\frac{1}{n}}$ strongly in $L^1( \|V\|)$.
    Applying the inequality~\eqref{eqn:Ak-A-inequality}, the right-hand side is uniformly integrable and $A_k(x) \to A(x)$ in measure, so Vitali's theorem implies that $D^2_V U_k(x) \to |B^n|^{\frac{1}{n}} \,\textup{id}_{T_x V}$ strongly in $L^1( \|V\|)$.

    We now recall the argument of Propositions~\ref{prop:fancy-integration-by-parts} and~\ref{prop:measure-zeta-q}: writing $\delta_s V = \zeta \, |\delta_s V|$ for $\zeta$ a $\bR^{n+2}$-valued function with $|\zeta|=1$ for $|\delta_s V|$-a.e.~ points, we obtain for each $k$ a signed Radon measure $\lambda_k$ and a function $q_k \in L^{\infty}_{\textup{loc}}( |\delta_s V| ; \bR^{n+2})$ such that $|q_k| \leq 1$ and $\lambda_k - \Delta_V U_k \, \|V\| \geq 0$, while
    \begin{equation}\label{eqn:IBP-for-lambda-k}
    \int \psi \, d \lambda_k = - \int \la \nabla^V \psi, \nabla^V U_k \rg \, d\|V\| + \int \psi \la q_k , \zeta \rg \, d |\delta_s V| \qquad \text{for all } \; \psi \in \textup{Lip}_c (\Omega),
    \end{equation}
    where $\Omega$ is a bounded open neighborhood of $\textup{spt} \, \|V\|$.
    Taking a sequence of $\psi \in C_c^1(\Omega)$ converging to $1$ on $\on{spt} \|V\|$, we obtain $\lambda_k(\Omega) = \int \la q_k, \zeta \rg \, d |\delta_s V| \leq |\delta_s V|(\Omega)$.
    Thus,
    \[
    |\delta_s V|(\Omega) - \int \Delta_V U_k \, d \|V\| = \Bigl[ \lambda_k(\Omega) - \int \Delta_V U_k \, d \|V\| \Bigr] + \bigl[ |\delta_s V| (\Omega) - \lambda_k(\Omega) \bigr]
    \]
    where both terms on the right-hand side are non-negative, and the left-hand side tends to zero by Step 1.
    Therefore, $(\lambda_k - \Delta_V U_k \|V\|)( \Omega) \to 0$ and $\int (1 - \la q_k, \zeta \rg) \, d |\delta_s V| \to 0$.
    Also,
    \[
    |q_k - \zeta|^2 = |q_k|^2 + |\zeta|^2 - 2 \la q_k, \zeta \rg = |q_k|^2 + 1 - 2 \la q_k, \zeta \rg \leq 2 ( 1 - \la q_k, \zeta \rg).
    \]
    Thus, $q_k \to \zeta$ strongly in $L^2( |\delta_s V|)$, hence also in $L^1( |\delta_s V|)$ due to $|\delta_s V|(\bR^{n+2}) < \infty$.
    Moreover, the measures $\lambda_k - \Delta_V U_k \, \|V\| $ and their total masses tend to zero, so $\lambda_k - \Delta_V U_k \, \|V\| \to 0$ in total variation.
    Combining this with the strong $L^1(\|V\|)$ convergence of $\Delta_V U_k$, we obtain $\lambda_k \to n |B^n|^{\frac{1}{n}} \|V\|$ in total variation, hence also in the weak-$\ast$ sense as Radon measures.
    We send $k \to \infty$ and pass to the limit in the integration by parts identity~\eqref{eqn:IBP-for-lambda-k} for $\lambda_k$, using $\nabla^V U_k \to \nabla^V U$ strongly in $L^1(\|V\|)$ and $q_k \to \zeta$ strongly in $L^1 (|\delta_s V|)$.
    This proves the identity~\eqref{eqn:equality-limit-psi-theta(x)} for every $\psi \in \textup{Lip}_c(\Omega)$.

    \smallskip \noindent \textbf{Step 4: The Hessian of $U$.}
    We now study $D^2_VU$.
    Fix a Borel set $\cB_j \subset M_j$ and a relatively compact coordinate patch $W \Subset M_j$.
    Becuase $\mathrm{I\!I}_{M_j}=0$ $\|V\|$-a.e.~ and the functions $U_k$ are convex and $1$-Lipschitz in ambient space, their restrictions to $W$ are uniformly semiconvex, so there is a constant $C_W$ such that the tensor-valued Radon measures $D^2_{M_j} (U_k|_{M_j}) + C_W g_{M_j} \cH^n \mres W$ are positive semidefinite.
    Moreover, their masses are uniformly bounded as in Proposition~\ref{prop:fancy-integration-by-parts}, with 
    \[
    |D^2_{M_j} (U_k|_{M_j}) + C_W g_{M_j} \, \cH^n \mres M_j|(K) \leq \la \Delta_{M_j} (U_k|_{M_j}) + n C_W, \varphi_K \rg
    \]
    for every compact set $K \Subset W$ and some non-negative $\varphi_K \in C^{\infty}_c(W)$ with $\varphi \equiv 1$ near $K$.
    The right-hand side is uniformly bounded because $U_k \to U$ locally uniformly.
    Thus, we can apply weak-$\ast$ compactness and use the uniqueness of the distributional limit to obtain
    \begin{equation}\label{eqn:hessian-convergence}
    D^2_{M_j}(U_k|_{M_j}) + C_W g_{M_j} \, \cH^n \mres M_j \xrightharpoonup{*} D^2_{M_j}(U|_{M_j}) + C_W g_{M_j} \,  \cH^n \mres M_j.
\end{equation}
On $\cB_j$, we have $d \|V\| = d \cH^n$ for $\cH^n$-a.e.~ point due to $\Theta^n ( \|V\|,x) = 1$ for $\|V\|$-a.e.~ $x$, while $D^2_V U_k \to |B^n|^{\frac{1}{n}} \, \textup{id}_{T_x V}$ from Step 2; hence, the property~\eqref{eqn:hessian-convergence} implies that $(D^2_V U_k + C_W g_{M_j}) \, \cH^n \to (|B^n|^{\frac{1}{n}} + C_W) g_{M_j} \, \cH^n$ in total variation.
For semiconvex functions, the absolutely continuous density of the distributional Hessian is the Alexandrov Hessian, so
\begin{equation}\label{eqn:D2UkM_j}
D^2_{M_j} (U_k|_{M_j}) + C_W g_{M_j} \, \cH^n - ( D^2_V U_k + C_W g_{M_j}) \, \cH^n \mres ( \cB_j \cap W) \geq 0
\end{equation}
as positive semidefinite tensor-valued measures. 
We may now pass to the limit as $k \to \infty$ in~\eqref{eqn:D2UkM_j}: the first sequence of measures converge due to~\eqref{eqn:hessian-convergence}, while the subtracted absolutely continuous measures converge in total variation, hence also weak-$\ast$, by the integral bound.
We therefore obtain $D^2_{M_j} (U|_{M_j}) \geq |B^n|^{\frac{1}{n}} \, g_{M_j}$ on $\cB_j \cap W$, hence $D^2_V U \geq |B^n|^{\frac{1}{n}} \, \textup{id}_{T_x V}$ for $\|V\|$-a.e.~ $x$.

On the other hand, the equality~\eqref{eqn:equality-limit-psi-theta(x)} allows us to apply Proposition~\ref{prop:measure-zeta-q} with $\lambda := n |B^n|^{\frac{1}{n}} \|V\|$ and $q := \zeta$ and obtain $n |B^n|^{\frac{1}{n}} \|V\| - \Delta_V U \, \|V\| \geq 0$, so $\Delta_V U \leq n |B^n|^{\frac{1}{n}}$ holds $\|V\|$-a.e.
Moreover, taking traces in $D^2_V U \geq |B^n|^{\frac{1}{n}} \, \textup{id}_{T_x V}$ proves the reverse bound, since $H_V(x)=0$ for $\|V\|$-a.e.~ $x$, so $\Delta_V U = n |B^n|^{\frac{1}{n}}$ for $\|V\|$-a.e.~ $x$.
Thus, the symmetric tensor $D^2_V U - |B^n|^{\frac{1}{n}} \, \textup{id}_{T_x V} \geq 0$ is positive semidefinite and has zero trace for $\|V\|$-a.e.~ $x$, so it vanishes; this proves that $D^2_V U = |B^n|^{\frac{1}{n}} \textup{id}_{T_x V}$.

    \smallskip \noindent \textbf{Step 5: The subdifferential $\partial U$.}
Finally, we show that the limiting optimal map $U$ takes the full normal disk into the Euclidean subdifferential of $U$.
To show~\eqref{eqn:subdifferential-U} at jointly good points, it suffices to prove that $\nabla^V U(x) + y \in \partial U(x)$ for every $y \in T_x^{\perp} V$ with $|y| = \sqrt{1 - |\nabla^V U(x)|^2}$: since $\partial U(x)$ is closed and convex, this will imply that $\nabla^V U(x) + \bar{B}_{\sqrt{1 - |\nabla^V U(x)|^2}}(0) \cap T_x^{\perp} V \subset \partial U(x)$.
The reverse inclusion holds because $U$ is $1$-Lipschitz, so $|\xi| \leq 1$ and $|\xi^{\perp}| \leq \sqrt{1 - |\nabla^V U(x)|^2}$ for any $\xi \in \partial U(x)$.

To prove this weaker statement, we argue by contradiction and suppose that there is a good point $x_0$ with $p_0 := \nabla^V U(x_0)$ and a $y_0 \in T_{x_0}^{\perp} V \cap \partial B_{\sqrt{1- |p_0|^2}}(0)$ with $\xi_0 := p_0 +y_0 \not\in \partial U(x_0)$, so $B_{\eta_0}(\xi_0) \cap \partial U(x_0) = \varnothing$ for $\eta_0>0$, because $\partial U(x_0)$ is closed and convex.
Now, $x_0$ is a good point, and the locally uniform convergence $U_k \to U$ of convex $1$-Lipschitz functions implies that their subdifferentials converge in the Kuratowski sense, cf.~\cite{rockafellar}*{Theorem 24.5}.
Therefore,
\begin{equation}\label{eqn:Uk-subdifferential-contradiction}
\exists \; \{ r_0 > 0, \; \eta_1 >0 , \; k_0\}, \qquad U_k(x) - h_k(\xi) - \la x, \xi \rg > 0 \quad \text{for all } \; k \geq k_0, 
\end{equation}
whenever $x \in Q_r \cap \on{spt} \|V\|$ with $r \in (0,r_0)$ and $\xi \in B_{\eta_1}(\xi_0)$.
We let $E_r := Q_r \cap \cB_j \cap \on{spt} \|V\|$ and define the sets $A_{r,k} , G_{r,k}$ as in Proposition~\ref{prop:laplacian-inequality} in terms of smooth second-order representatives $\hat{U}_k$ of $U_k|_{M_j}$ at $x_0$.
Then, the argument used therein shows that all contacts of $(U_k, h_k)$ over $E_r$ have second coordinate in $G_{r,k}$, while the assumption~\eqref{eqn:Uk-subdifferential-contradiction} implies that no contact can lie in $B_{\eta_1}(\xi_0)$.
Therefore, Lemmas~\ref{lemma:E-Sigma-perturb} and~\ref{lemma:BE-lemma-7}, with $G$ replaced by $G_{r,k} \setminus B_{\eta_1}(\xi_0)$, imply that $\|V\|(E_r) \leq \nu_k ( G_{r,k} \setminus B_{\eta_1}(\xi_0))$.

We now estimate the right-hand side.
Fix $\eta \in (0, \frac{1}{10} \eta_1)$ and let $\chi \in C^{\infty}( \bR^{n+2})$ be a non-negative function satisfying $\chi \equiv 0$ on $B_{\eta}(\xi_0)$ and $\chi \equiv 1$ on $\bR^{n+2} \setminus B_{2 \eta}(\xi_0)$, with $0 \leq \chi \leq 1$ and $\chi' \geq 0$, so
\begin{equation}\label{eqn:VEr-bound}
\|V\| (E_r) \leq \nu_k (G_{r,k} \setminus B_{\eta_1}(\xi_0)) \leq \int_{G_{r,k}} \chi(\xi) \, d \nu_k(\xi).
\end{equation}
As in Step 1 of Proposition~\ref{prop:laplacian-inequality}, we consider, for each $k$, the normal coordinate map
\[
\Psi_k(v,y) = \nabla^{M_j} \hat{U}_k (\Xi^{M_j}_{x_0} (v)) + \tau_v^{M_j}(\eta) , \qquad \text{for} \quad \tau_v^{M_j}: T_{x_0}^{\perp} V \to T_{\Xi^{M_j}_{x_0}(v)}^{\perp} M_j \quad \text{a $C^1$ map},
\]
with $\Psi_k(0,y) = p_k + y$ and $J \Psi_k(0,y) = \det D^2_V U_k(x_0)$ due to $\mathrm{I\!I}_V (x_0) =0$.
Using $H_V = 0$, $\mathrm{I\!I}_V = 0$, and $D^2_V U_k(x_0) \to |B^n|^{\frac{1}{n}} \, \textup{id}_{T_x V}$, $\nabla^V U_k(x_0) \to p_0$, we can apply the local coarea estimate from Step 1 of Proposition~\ref{prop:laplacian-inequality} with weight $\hat{\rho}_k(\xi) = \chi(\xi) \rho_k(|\xi|^2)$ to obtain, for $p_k := \nabla^V U_k(x_0)$,
\[
\limsup_{r \downarrow 0} r^{-n} \int_{G_{r,k}} \chi(\xi) \, d\nu_k(\xi) \leq \int_{S_k(x_0)} \chi(p_k + y) \det (D^2_V U_k (x_0) - \la \mathrm{I\!I}_V(x_0),y \rg ) \, \rho_k ( |p_k|^2 + |y|^2) \, dy.
\]
Because $\mathrm{I\!I}_V(x_0) = 0$, we have $S_k (x_0) = \{ |p_k|^2 + |y|^2 \leq 1\}$ and hence, using the bound~\eqref{eqn:VEr-bound} together with $\|V\|(E_r) = \theta(x_0) r^n + o(r^n)$ by Lemma~\ref{lemma:v-ae-q-gj}, we obtain
\[
1 \leq \alpha_k \det D^2_V U_k(x_0) \int \chi(p_k + y) \, d \tilde{\nu}_k(y), \qquad d \tilde{\nu}_k(y) = \alpha_k^{-1} \, \rho_k ( |p_k|^2 + |y|^2) \, \mathbf{1}_{ \{ |p_k|^2 + |y|^2 \leq 1\} } \, dy.
\]
By Lemma~\ref{lemma:converge-to-uniform-measure}, the fiber measures $\alpha_k^{-1} \rho_k (|p_k|^2 + |y|^2) \, \mathbf{1}_{ \{ |p_k|^2 + |y|^2 \leq 1 \} } \, dy$ converge to the uniform probability measure $\tilde{\nu}_{\infty}$ on the circle $\{ y \in T_{x_0}^{\perp} V : |y| = \sqrt{1 - |p_0|^2} \}$, and the set $\{ y : p_0 + y \in B_{\eta}(\xi_0) \}$ contains a non-empty open arc of this circle.
Because $\chi =0$ in a neighborhood of $\xi_0 = p_0 + y_0$, hence on such a non-empty arc, we obtain $\int \chi (p_0 + y) \, d \tilde{\nu}_{\infty}(y) < 1$.
Moreover, the convergence $\alpha_k^{-1} \to |B^n|$, $D^2_V U_k(x_0) \to |B^n|^{\frac{1}{n}} \textup{id}_{T_x V}$, and $p_k \to p_0$ implies that $\alpha_k \det D^2_V U_k(x_0) \to 1$, hence
\[
1 \leq \limsup_{k \to \infty} \alpha_k \det D^2_V U_k(x_0) \int \chi(p_k + y) \, d \tilde{\nu}_k(y) \leq \int \chi(p_0 +y) \, d \tilde{\nu}_{\infty}(y) < 1,
\]
a contradiction.
Since the data $(x_0, p_0, y_0)$ were arbitrary, the characterization~\eqref{eqn:subdifferential-U} follows.
\end{proof}

\begin{lemma}\label{lemma:zeta-limit}
    We express the singular part of the first variation measure as $\delta_s V = \zeta \, |\delta_s V|$, for $\zeta$ a $\bR^{n+2}$-valued function with $|\zeta|=1$ for $|\delta_s V|$-a.e.~ points.
    Then, $\partial U(x) = \{ \zeta(x) \}$ for $|\delta_s V|$-a.e.~ $x$.
    Moreover, for $|\delta_s V|$-a.e.~ $x_0$, any sequence of good points $\{ x_i \}$ with $x_i \to x_0$ satisfies $\nabla^V U(x_i) \to \zeta(x_0)$ and $\zeta(x_0)$ lies on every subsequential limit of the tangent spaces $T_{x_i}V$.
\end{lemma}
\begin{proof}
Mollifying $U_{\ve} := U \ast \chi_{\ve}$, the argument used in Proposition~\ref{prop:measure-zeta-q} and in Step 3 of Proposition~\ref{prop:reduce-to-nice-configuration} gives, after passing to a subsequence $\ve_k \downarrow 0$, $DU_{\ve_k} \xrightharpoonup{*} q \in L^{\infty}_{\textup{loc}} ( |\delta_s V|; \bR^{n+2})$ and $q = \zeta$ $|\delta_s V|$-a.e.
For a fixed point $x$, let $h_x(y) := \sup_{\xi \in \partial U(x)} \la \xi, y \rg$ be the support function of $\partial U(x) \subset \bar{B}_1^{n+2}$.
Because $U$ is convex and $1$-Lipschitz, the function $h_x$ is Lipschitz, hence differentiable a.e., with $\nabla h_x(-y) \in \partial U(x)$ for a.e.~ $y$.
Moreover, the rescaled functions $U_{x,r}(y) := \frac{U(x+ry) - U(x)}{r}$ are convex, $1$-Lipschitz, and satisfy $U_{x,r} \to h_x$ locally uniformly as $r \downarrow 0$, while $D U_{\ve}(x) = \int_{\bR^{n+2}} \frac{U(x-\ve y) - U(x)}{\ve} D \chi(y) \, dy$, with integrand converging uniformly on $\on{spt} \chi$ to $h_x(-y) D \chi(y)$.
Therefore,
\[
DU_{\ve}(x) \to \int_{\bR^{n+2}} h_x(-y) D \chi(y) \, dy = D (h_x \ast \chi)(0) = \int_{\bR^{n+2}} \nabla h_x(-y) \chi(y) \, dy =: S(x) \in \partial U(x).
\]
Since $|DU_{\ve}| \leq 1$, the dominated convergence theorem gives $DU_{\ve} \to S(x)$ strongly in $L^1(K , |\delta_s V|)$, for every compact $K \Subset \Omega$, so $q = \zeta = S$ for $|\delta_sV|$-a.e.~ $x$.
The convolution formula for $S(x)$ leads to
\[
1 = |\zeta(x)| =   \Bigl| \int_{\bR^{n+2}} \nabla h_x(-y) \chi(y) \, dy \Bigr| \leq \int_{\bR^{n+2}} |\nabla h_x(-y)| \chi(y) \, dy \leq 1
\]
due to $\nabla h_x(-y) \in \partial U(x) \in \bar{B}_1$ for a.e.~ $y$.
Thus, equality holds in both steps and $\nabla h_x(-y) = \zeta(x)$ for $\chi(y) \, dy$-a.e.~ $y$, so $\nabla h_x(y) = \zeta(x)$ for a.e.~ $y$ in a ball $B_{\rho}(0)$ with $\chi>0$.
Because $h_x$ is Lipschitz and $1$-homogeneous, this implies $h_x(y) = \la \zeta(x), y \rg$ for $y \in \bR^{n+2}$.
Moreover, any $\xi \in \partial U(x)$ has $\la \xi, y \rg \leq h_x(y) = \la \zeta(x) , y \rg$ for every $y$, so $\xi = \zeta(x)$ and $\partial U(x) = \{ \zeta(x) \}$ for $|\delta_s V|$-a.e.~ $x$.

For the second assertion, we use $|\nabla^V U| \leq 1$ at good points and the compactness of the Grassmannian $\bG(n+2,n)$ to extract a subsequence with $\nabla^V U(x_i) \to p$ and $T_{x_i} V \to P$ for some $P \in \bG(n+2,n)$ and a vector $p \in P$.
At every good point, $\nabla^V U(x_i) \in \partial U(x_i)$, and because the graph of the subdifferential of a convex function is closed, we obtain $p \in \partial U(x_0) = \{ \zeta(x_0)\}$.
Therefore, $p = \zeta(x_0)$ is the unique subsequential limit, and $\nabla^V U(x_i) \to \zeta(x_0) \in P$.
This completes the proof.
\end{proof}

With these properties at hand, we complete the proof of Corollary~\ref{cor:sharp-isoperimetric}.
\begin{proof}[Proof of Corollary~\ref{cor:sharp-isoperimetric} (i)]
    Suppose, additionally, that the tangent plane map $P :x \mapsto T_x V \in \bG(n+2,n)$ is weakly differentiable on $V$.
    In what follows, we abbreviate $c_n := |B^n|^{\frac{1}{n}}$ and write $V \mathbf{D} P$ for the weak varifold derivative as in~\cite{weakly-differentiable-menne}*{Definition 8.3}.
    By Menne's weak differentiability theorem~\cite{weakly-differentiable-menne}*{Theorem 11.2}, this expression is equal to the approximate derivative $V \, \textup{ap} \, DP$.
    On each $C^2$ sheet $M_j$ representing $V$, the map $P_{M_j}(x) = T_x M_j$ satisfies, for each $u \in T_x M_j$,
    \[
    (D^{M_j} P_{M_j}(x)[u] ) v = \mathrm{I\!I}_{M_j} (x) (u,v) \; \text{ for } \; v \in T_x M_j, \quad (D^{M_j} P_{M_j}(x) [u]) \eta = A_{\eta}(x) u \; \text{ for } \; \eta \in T_x^{\perp} M_j.
    \]
    Here, $A_{\eta}$ denotes the shape operator, defined by $\la A_{\eta} u, v \rg = \la \mathrm{I\!I}_{M_j}(u,v), \eta \rg$ for $u,v \in T_x M_j$.
    Therefore, taking $\{ e_1, \dots, e_n \}$ to be an orthonormal basis for $T_x M_j$, we obtain
    \[
    |D^{M_j}P_{M_j}(x) [u]|^2 = 2 \sum_{i=1}^n |\mathrm{I\!I}_{M_j}(x) (u,e_i)|^2, \qquad \text{and hence} \quad |D^{M_j} P_{M_j}(x)|^2 = 2 \, |\mathrm{I\!I}_{M_j}(x)|^2.
    \]
    Thus, for $\|V\|$-a.e.~ good point $x_0$, which is a density point of $\cB_j \subset M_j$ with $T_{x_0} V = T_{x_0} M_j$ and $\mathrm{I\!I}_{M_j}(x_0) = \mathrm{I\!I}_V(x_0) = 0$, we have $P_{M_j}(x) = P_{M_j}(x_0) + o ( |x - x_0|)$ as $x \to x_0$ on $M_j$.
    On the density-one set $\cB_j$, we have $P(x) = P_{M_j}(x)$ for $\|V\|$-a.e.~ $x$, hence for every $\ve>0$, we obtain
    \[
    \Theta^n \bigl( \|V\| \mres \{ x : |P(x) - P(x_0)| > \ve |x - x_0| \}, x_0 \bigr) = 0.
    \]
    This property implies that $V \, \textup{ap} \, DP = 0$ for $\|V\|$-a.e.~ such $x_0$, hence $V \mathbf{D} P = 0$ $\|V\|$-a.e.
By Menne's constancy result~\cite{weakly-differentiable-menne}*{Theorem 8.33} and the fact that $V$ is indecomposable, due to Lemma~\ref{lemma:indecomposable}, we obtain $P(x) = \Pi$ for $\|V\|$-a.e.~ $x$ and a fixed $n$-plane $\Pi$.

For any vector $e \in \Pi^{\perp}$, we have $\nabla^V \la x, e \rg = \pi_{T_x V} e = \pi_{\Pi} e = 0$ for $\|V\|$-a.e.~ $x$, so the constancy theorem and the indecomposability of $V$ imply that $\la x,e \rg$ is constant $\|V\|$-a.e.
Because this holds for every $e \in \Pi^{\perp}$, we obtain $\textup{spt} \,\|V\| \subset a + \Pi$ up to $\|V\|$-null sets, for some fixed $a \in \bR^{n+2}$.
Since $\Theta^n (\|V\|,x)=1$ a.e., we have $V = \cH^n \mres E$ for a measurable set $E \subset a + \Pi \simeq \bR^n$, which is a Caccioppoli set because $\delta V$ is a finite Radon measure with vanishing absolutely continuous mean curvature.
Moreover, $\|V\| (\bR^{n+2}) = \cH^n(E) = 1$ and $|\delta_s V| = |D \chi_E|$ with $|\delta_s V|( \bR^{n+2}) = \cH^{n-1}(\partial^*E) = n c_n$.

Now, let $u := U|_{\Pi}$, which is a convex $1$-Lipschitz function with $\nabla^{\Pi} u$ a BV vector field whose derivative is the Hessian measure $D^2 u$.
By convexity, $D^2 u$ is a positive semidefinite matrix-valued measure.
On $E$, we have $\nabla^V U = \nabla^{\Pi} u$ $\cH^n \mres E$-a.e., so the identity $D^2_V U = c_n \textup{id}_{T_x V}$ implies that $D^2 u = c_n \textup{id}_{\Pi} + (D^2 u)_s$, for $(D^2 u)_s \geq 0$ a non-negative singular measure.
Using~\eqref{eqn:equality-limit-psi-theta(x)}, we find
\[
\textup{div} ( \chi_E \nabla^{\Pi} u) = n c_n \chi_E \, \cL^n - |\delta_s V|, \qquad |\delta_s V| = |D \chi_E|
\]
as distributions.
Restricting to the measure-theoretic interior $E^{(1)}$ and using the fact that $|\delta_s V| = |D \chi_E|$ is concentrated on the measure-theoretic boundary, we obtain $\Delta u \mres E^{(1)} = n c_n \cL^n \mres E$, and hence $(D^2 u)_s = 0$ on $E^{(1)}$.
Thus, $D^2 u \mres E^{(1)} = c_n \textup{id}_{\Pi} \, \cL^n \mres E$ and $D^2_V U = V \mathbf{D} (\nabla^V U)$ agrees with the varifold weak derivative, so $D^2_V U = c_n \, \textup{id}_{T_x V}$ implies $V \mathbf{D} ( \nabla^V U - c_n x) = 0$ and $\nabla^V U = c_n( x - p)$ for some $p \in \bR^{n+2}$, by~\cite{weakly-differentiable-menne}*{Theorem 8.33}.
Since $|\nabla^V U|<1$ for $\|V\|$-a.e.~ $x$, we find $|x-p| < c_n^{-1}$, so $\bar{E} = \textup{spt} \, \|V\| \subset \bar{B}^{\Pi}_{c_n^{-1}}(p)$ up to an $\cH^n$-null set.
Also, $1 = \cH^n(E) = \cH^n( B^{\Pi}_{c_n^{-1}}(p))$, so $E = B^{\Pi}_{c^{-1}_n}(p)$ up to an $\cH^n$-null set and $V = |B^{\Pi}_{c_n^{-1}}(p)|$ is a flat round $n$-disk.
This completes the proof.
\end{proof}

\begin{proof}[Proof of Corollary~\ref{cor:sharp-isoperimetric} (ii)]
Let $\Sigma := \textup{sing} \, V$ denote the set of points at which $V$ is not locally a smooth multiplicity-one manifold.
Let $\cR := \textup{spt} \, \|V\| \setminus \Sigma$ denote the regular set of $V$, which is open, and write $\cG \subset \cR$ for the set of good points that satisfy the properties of Proposition~\ref{prop:reduce-to-nice-configuration} and are density points for $\cR$, with $\|V\|(\cG) = \|V\|(\cR)$.
For each $x_0 \not\in \Sigma$, there exists a $\rho_0>0$ such that $V = |M_0|$ on $B_{\rho_0}(x_0)$, for $M_0$ a smooth graph.
Because $\mathrm{I\!I}_V = 0$ a.e., this implies that each connected component of $\textup{spt} \, \|V\| \cap B_{\rho_0/2}(x_0)$ is an open subset of an affine $n$-plane $a+\Pi$, hence $\| \delta V\| ( B_{\rho_0/2}(x_0)) = 0$ and $x_0 \not\in \textup{spt} \, \| \delta V\|$.
Moreover, the argument of Corollary~\ref{cor:sharp-isoperimetric}$\,(i)$ shows that $D_{\Pi}( \nabla^{\Pi} (U|_{\Pi}) - c_n x) = 0$ on each such plane, hence $\nabla^V U = \nabla^{\Pi} U = c_n (x - p_{\Pi})$ in $B_{\rho_0}(x_0)$ for some $p_{\Pi} \in a+ \Pi$.
Combined with the condition $\cH^{n-1} \mres \Sigma \ll \| \delta V\|$  and the fact that $H_V = 0$ a.e., this implies that $\cH^{n-1} ( \Sigma \setminus \textup{spt} \, |\delta_s V|) = 0$.
Because $\| \delta V\| = n c_n $, with $c_n = |B^n|^{\frac{1}{n}}$ and $\|V\| (\bR^{n+2}) = 1$, we obtain $\cH^{n-1}(\Sigma) < \infty$ and $\|V\| (\Sigma) = \cH^n(\Sigma) = 0$.
Therefore, $\|V\|(\cG) = \|V\|( \bR^{n+2})=1$ and the connected components $M_j$ of $\cR$ cover $V$ up to $\|V\|$-measure zero.

The set $\cR$ is a second-countable, locally connected smooth $n$-manifold, so its connected components $\{ M_j \}$ form a countable family and each $M_j \subset \Pi_j$ is contained in an $n$-plane, with $\nabla^V U = c_n (x - p_j)$ for some $p_j \in \Pi_j$ and each $x \in M_j$.
Because $|\nabla^V U| <1$ for $\|V\|$-a.e.~ $x$, we have $M_j \subset \Pi_j \cap B_{c_n^{-1}}(p_j) =: B^{\Pi_j}_{c_n^{-1}} (p_j)$ as sets.
Moreover, any $x \in \partial_{\Pi_j} M_j$ has $x \in \bar{M}_j \subset \textup{spt} \, \|V\|$, so $x \not\in \cR$, else there would exist some $\rho>0$ with $B_{\rho}(x) \subset M_j$, contradicting the fact that $x$ is a boundary point.
Thus, $\cH^{n-1} ( \partial_{\Pi_j} M_j) \leq \cH^{n-1} (\Sigma)< \infty$, and since the essential boundary satisfies $\partial^e_{\Pi_j} M_j \subset \partial_{\Pi_j} M_j$, Federer's criterion~\cite{federer-gmt}*{4.5.11} implies that $M_j$ is a Caccioppoli set in $\Pi_j$.

We now define $\sigma_j := \cH^{n-1} \mres \partial^*_{\Pi_j} M_j$.
Because $\partial^*_{\Pi_j} M_j \subset \Sigma$, the assumed absolute continuity gives $\sigma_j \ll \cH^{n-1} \mres \Sigma \ll |\delta_s V|$.
We fix a point $x \in \partial^*_{\Pi_j} M_j$ outside a $\sigma_j$ null set, such that $|\zeta|=1$ Lemma~\ref{lemma:zeta-limit} applies.
Because the good points have full measure in $M_j$, we can take a sequence of good points $\{ x_k \} $ in $M_j$ with $x_k \to x$, so Lemma~\ref{lemma:zeta-limit} implies that $c_n (x_k - p_j) = \nabla^V U(x_k) \to \zeta(x)$.
Thus, $\zeta(x) = c_n (x - p_j)$ and $c_n |x-p_j|=1$.
Since $M_j \subset B^{\Pi_j}_{c_n^{-1}}(p_j)$ and $x$ lies on the boundary sphere, while $M_j$ is a Caccioppoli set, the reduced blow-up of $M_j$ at $x$ is a half-space through the origin contained in the tangent half-space of the ball, by De Giorgi's blow-up theorem~\cite{maggi}*{Theorem 15.5}.
Both half-spaces have density $\frac{1}{2}$, so because $M_j \subset B^{\Pi_j}_{c_n^{-1}}(p_j)$, they coincide and the outward conormal $\nu_j$ of $M_j$ agrees with the outward conormal of the ball.
This means that
\begin{equation}\label{eqn:nu-j(x)}
\nu_j(x) = \tfrac{x - p_j}{|x - p_j|} = c_n (x - p_j) = \zeta(x), \qquad \text{for } \; \sigma_j\text{-a.e.~ } \; x.
\end{equation}
Therefore, $\la \nabla^{\Pi_j} U, \nu_j \rg = 1$ at $\cH^{n-1}$-a.e.~ reduced boundary point of every regular component $M_j$.
Thus, applying Green's theorem in each $\Pi_j$ for a non-negative $\varphi \in C_c^1( \bR^{n+2})$, we obtain
\begin{align*}
    \int_{\partial^*_{\Pi_j} M_j} \varphi \, d \cH^{n-1} &= \int_{\partial^*_{\Pi_j} M_j} \varphi \la \nabla^V U, \nu_j \rg \, d \cH^{n-1} = \int_{M_j} \textup{div}_{\Pi_j} ( \varphi \nabla^{\Pi_j} U ) \, d \cH^n \\
    &= \int_{M_j} ( \la \nabla^{\Pi_j} \varphi, \nabla^{\Pi_j} U \rg + n c_n \varphi ) \, d \cH^n.
\end{align*}
In the last step, we used the fact that $\Delta_V U = n c_n$.
Summing over the regular components $M_j$, we observe that the right-hand side converges absolutely, because $M_j$ are pairwise disjoint up to $\cH^n$-null sets and $|\la \nabla^{\Pi_j} \varphi, \nabla^{\Pi_j} U \rg| \leq |\nabla \varphi| \leq C_{\varphi}$.
Because $\varphi \geq 0$, the left-hand side is monotone as more components are included.
Thus, the two series are absolutely summable and
\[
\sum_j \int_{\partial^*_{\Pi_j} M_j} \varphi \, d \cH^{n-1} = \int ( \la \nabla^V \varphi, \nabla^VU \rg + n c_n \varphi) \, d \|V\| = \int \varphi \, d |\delta_s V|
\]
by applying the identity~\eqref{eqn:equality-limit-psi-theta(x)}.
Therefore, $|\delta_s V| = \sum_j \sigma_j$ as Radon measures.
We show that the $\sigma_i$ are pairwise mutually singular: if $\cH^{n-1} ( \partial^*_{\Pi_i} M_i \cap \partial^*_{M_j} M_j) > 0$ for some $i \neq j$, then at a.e.~ point $x$ of this intersection, the reduced boundaries have the same approximate tangent $(n-1)$-space $H$.
Then,~\eqref{eqn:nu-j(x)} implies that the outward conormals $\nu_i = \nu_j = \zeta$ agree, so $T \Pi_i = H \oplus \bR \zeta = T \Pi_j$ shows that the ambient planes agree.
Both affine planes pass through $x$, so $\Pi_i = \Pi_j$, and since $\nu_i = \nu_j$, the reduced blowups of both $M_i, M_j$ are the same half-space $\{ y \in \Pi_i : \la y, \zeta(x) \rg < 0 \}$.
Therefore,
\[
\cH^n( M_i \cap M_j \cap B_r(x)) = \tfrac{1}{2} |B^n| r^n + o(r^n) \qquad \text{as } \; r \downarrow 0,
\]
contradicting $\cH^n (M_i \cap M_j) = \varnothing$.
We conclude that $\sigma_i \perp \sigma_j$ for $i \neq j$.
Since the countable family $\{ \sigma_j \}$ is mutually singular and $\sum_j \sigma_j = |\delta V|$, there are pairwise disjoint Borel sets $\cB_j \subset \Sigma$ satisfying $\sigma_j = |\delta_s V| \mres \cB_j$ and $|\delta V| ( \Sigma \setminus \bigcup_j \cB_j) = 0$.
We define $E_j := M_j \cup \cB_j$, so $\|V\|(\Sigma) = 0$ implies that the varifolds $V_j := \cH^n \mres M_j$ have first variation $\delta V_j = \nu_j \sigma_j = \zeta \sigma_j$ and
\[
V \mres (E_j \times \bG(n+2,n) ) = V_j, \qquad ( \delta V) \mres E_j = \zeta \, |\delta V| \mres \cB_j = \zeta \sigma_j = \delta V_j.
\]
Thus, with the definition~\cite{weakly-differentiable-menne}*{5.1} of Menne, $V \partial E := ( \delta V) \mres E - \delta \bigl( V \mres ( E \times \bG(n+2,n)) \bigr)$, we obtain $V \partial E_j = 0$ for every $j$.
The varifolds $V_j$ produce a decomposition of $V$ in the sense of~\cite{weakly-differentiable-menne}*{Definition 6.2}.
Applying Lemma~\ref{lemma:indecomposable}, $V$ is indecomposable, so there is exactly one index $k$ with $\cH^n (M_k) > 0$.
Therefore, $V = \cH^n \mres M_k$ with $\partial^* M_k \subset \partial^* B_{c_n^{-1}}^{\Pi_k}(p_k)$ and $M \subset B_{c_n^{-1}}^{\Pi_k}(p_k)$.
Since $1 = \cH^n(M_k) = \cH^n ( B^{\Pi_k}_{c_n^{-1}}(p_k))$ and $n c_n = \cH^n(\partial^*_{\Pi_k} M_k) = \cH^n ( \partial B_{c_n^{-1}}^{\Pi_k}(p_k))$, we conclude that $M_k = B^{\Pi_k}_{c_n^{-1}}(p_k)$ up to an $\cH^n$-null set, so $V = |M_k|$ is a flat round $n$-disk.
\end{proof}

\bibliography{ref}

\end{document}